\documentclass[11pt]{article}
\usepackage{graphicx} 
\usepackage{graphicx}
\usepackage[dvipsnames]{xcolor}
\usepackage{amsmath}
\usepackage{lscape}
\usepackage{subfigure}
\usepackage{float}
\usepackage[center]{caption}
\usepackage{comment}

\usepackage{enumitem}
 \usepackage{booktabs}
 
\usepackage{natbib}
 \bibpunct[, ]{(}{)}{,}{a}{}{,}%
{\end{list}} 
\newenvironment{hangref}{\begin{list}{}
{\setlength{\leftmargin}{+\parindent}
\setlength{\itemindent}{-\parindent}}}{\end{list}} 

\usepackage{setspace}

\newcommand{\handout}[5]{
   \renewcommand{\thepage}{#1-\arabic{page}}
   \noindent
   \begin{center}
   \framebox{
      \vbox{
    \hbox to 5.78in { {\bf 18.405J/6.841J Advanced
  Complexity Theory} \hfill #2 }
       \vspace{4mm}
       \hbox to 5.78in { {\Large \hfill #5  \hfill} }
       \vspace{2mm}
       \hbox to 5.78in { {\it #3 \hfill #4} }
      }
   }
   \end{center}
   \vspace*{4mm}
}

\newtheorem{theorem}{Theorem}
\newtheorem{corollary}{Corollary}
\newtheorem{lemma}{Lemma}

\newtheorem{proposition}[theorem]{Proposition}

\newtheorem{remark}{Remark}

\newcommand{\blot}{\rule{7pt}{7pt}}

\newenvironment{proof-sketch}{\noindent{\bf Sketch of Proof}\hspace*{1em}}{\blot\bigskip}
\newenvironment{proof-idea}{\noindent{\bf Proof Idea}\hspace*{1em}}{\blot\bigskip}
\newenvironment{proof-of-lemma}[1]{\noindent{\bf Proof of Lemma #1}\hspace*{1em}}{\blot\bigskip}
\newenvironment{proof-attempt}{\noindent{\bf Proof Attempt}\hspace*{1em}}{\blot\bigskip}

\makeatletter
\def\fnum@figure{{\bf Figure \thefigure}}
\def\fnum@table{{\bf Table \thetable}}
\long\def\@mycaption#1[#2]#3{\addcontentsline{\csname
  ext@#1\endcsname}{#1}{\protect\numberline{\csname 
  the#1\endcsname}{\ignorespaces #2}}\par
  \begingroup
    \@parboxrestore
    \small
    \@makecaption{\csname fnum@#1\endcsname}{\ignorespaces #3}\par
  \endgroup}
\def\mycaption{\refstepcounter\@captype \@dblarg{\@mycaption\@captype}}
\makeatother

\newcommand{\mathify}[1]{\ifmmode{#1}\else\mbox{$#1$}\fi}

\newcommand{\bigO}O

\def\blot{\quad \mbox{$\vcenter{ \vbox{ \hrule height.4pt
      \hbox{\vrule width.4pt height.9ex \kern.9ex \vrule width.4pt}
           \hrule height.4pt}}$}}

\allowdisplaybreaks

\begin{document}

\begin{titlepage}

\vspace*{.2in}

\centerline{\LARGE\bf Vehicle Platooning}

\vspace{1.2in}
    
\centerline{\Large Zhi-Long Chen $^*$}

\vspace{.3in}

\centerline{\Large Nicholas G.\ Hall $^{\dagger}$}

\vspace{2.2in}

\noindent $^*$ Robert H. Smith School of Business, University of Maryland, College Park, MD 20742

\medskip

\noindent $^{\dagger}$ Fisher College of Business, The Ohio State University, 2100 Neil Avenue, Columbus, OH 43210

\vspace{2.8in}

\centerline{June 24, 2026} %

\end{titlepage}

\pagestyle{empty}

\centerline{\large\bf Abstract}

\vspace{.4in}

\noindent {\bf Abstract:} Vehicle platooning offers significant benefits, including reduced energy consumption, lower emissions, improved road utilization, enhanced safety, and reduced driver fatigue. As intelligent driving technologies continue to advance, platoon sizes are expected to increase substantially, making the efficient sequencing and resequencing of vehicles increasingly important. We study the vehicle platoon sequencing and resequencing problem on road networks with varying segment lengths under two fundamental objectives: minimizing total energy consumption and minimizing the maximum energy consumption of any vehicle. 
For the typically encountered combinations of vehicle and road characteristics, we provide a complete computational complexity classification, either developing polynomial-time algorithms or proving computational intractability. For several intractable cases, we design fully polynomial-time approximation schemes and polynomial-time heuristics with provable performance guarantees. A computational study demonstrates that the proposed heuristics achieve average solutions within 1\% of optimal. We also consider settings in which only limited information about position-dependent energy savings is available and develop a heuristic with bounded worst-case performance.  In addition, we present an efficient algorithm for on-road vehicle resequencing when only limited position changes are permitted. Together, these results provide a comprehensive algorithmic framework for energy-efficient vehicle platoon sequencing and resequencing.

\vspace{.2in}


\vspace{.2in}

\noindent {\bf 2020 Mathematics Subject Classification} \\
Primary: \\
90B06 — Transportation, logistics and supply chain management \\
Secondary: \\
90C27 — Combinatorial optimization \\
68Q25 — Analysis of algorithms and problem complexity \\
90C59 — Approximation methods and heuristics in mathematical programming

\newpage 

\pagestyle{plain}%
\pagenumbering{arabic}

\section{Introduction}  \label{sec:intro} 

A cooperative {\em vehicle platoon} is a fleet of vehicles traveling together, supported by advanced coordination and driving control technology. Well-documented benefits, resulting from reduction in aerodynamic drag, include savings in energy cost, environmental benefits, and road capacity usage, as well as reduction in driver fatigue and improved safety. A lead vehicle in a platoon may experience energy savings in the range of 4\% to 7\% for trucks. More significantly, a following vehicle may save between 10\% and 15\%. A study by Hussein and Rakha (2020) finds energy reductions of 4.5\% for light-duty vehicles, 15.5\% for buses, and 7.9\% for heavy-duty trucks. Improved inter-vehicle communication is enabling larger platoons and even greater savings. The amount of savings depends partly on fixed physical characteristics that constrain inter-vehicle distance or time, or travel speed.

However, two major decisions that affect energy savings are the (a) {\em initial sequencing} of the vehicles, and (b) their later {\em resequencing} within the platoon journey. We  focus on the two objectives of minimizing the total energy usage of the platoon and minimizing the maximum energy usage of any vehicle. 
 Most of the existing literature has developed heuristics without analyzing their theoretical performance. By contrast, we develop optimal algorithms and heuristic procedures for solving these problems, and evaluate their worst case performance analytically and their typical performance computationally.

The established benefits of vehicle platooning described above are motivating strong growth in its use. The truck platooning systems market size in the U.S. was \$223.44m in 2025, and is expected to grow to \$26.47b by 2035, with a 61.2\% CAGR (ResearchNester 2026).
For automotive platooning systems more generally, the market is estimated at \$6.6b in 2025 and is projected to reach \$48.3b by 2035, with heavy commercial vehicles representing 57.8\% of the market (Future Market Insights, Inc. 2026). It is estimated that system-level implementation in the U.S. can generate annual savings of \$868m for the trucking sector. Improved vehicle-to-vehicle communication, forward-looking radar, and intelligent braking systems will enhance the available benefits in two ways. First, the use of closer inter-vehicle spacing  within road safety parameters will magnify the energy saving of the following vehicles. Second, increases in feasible platoon size will improve the average energy saving per truck.  Under proposed automated highway systems, platoons of up to 25 vehicles are envisioned (Driver Knowledge Test Resource Center 2026), and simulation studies estimate potential energy savings of over 15\% at that platoon size (Lichtl\'{e} et al. 2024). This projected growth in platoon size makes the sequencing and resequencing problems which we consider both more valuable and more challenging.

Vehicle platooning is widely used for both internal combustion engine (ICE) vehicles and electric vehicles (EV). We briefly review some differences. EVs offer three additional gains not available to ICE vehicles. First, when braking, EVs run their electric motors in reverse, which recharges the battery (Recurrent Auto 2023). Second, EVs deliver full torque instantly when requested, which enables shorter spacing between vehicles (Zhang et al. 2025). Third, electric motors respond better to speed fluctuations than ICE vehicles (Recurrent Auto 2024). Hence, we adopt the notation of EVs such as state of charge (SoC). Nonetheless, all our mathematical results hold for vehicles of both types.

When a platoon contains heterogeneous vehicles, the sequence in which they travel affects the overall performance of the platoon, as measured for example by total energy usage. Platoon journeys typically consist of multiple road segments, with an opportunity to resequence the vehicles between them, for example at a truck stop. Resequencing can also occur on-road within prespecified zones that offer safety features such as multiple lanes, a lack of merging traffic, and an expectation of lower overall traffic volume. Intuitively, resequencing enables balancing the energy usage or the remaining SoC of the vehicles. This is needed when a vehicle with very low SoC may incur damage or even become inoperable. 

Heterogeneity between vehicles arises from their physical characteristics. Prominent among the relevant vehicle parameters are (a) the charge capacity, (b) the initial SoC, (c) the nominal energy usage rate if traveling individually, and (d) the minimum SoC that the vehicle can accept. Additionally, as observed above, the effective energy usage rate of a vehicle is determined not only by its nominal usage rate but also by its position in the platoon. Earlier positions typically define a higher effective usage rate. A further complication is that road segments within a platoon journey typically have different distances. This results from naturally occurring
variety in suitable locations for resequencing.
For example, a relatively long road segment may deplete the SoC of the leading vehicles, resulting in the need to resequence them into later positions for multiple later road segments. All these issues make the initial sequencing of the vehicles, and their later resequencing between road segments, complex optimization problems. Our work addresses these problems for a variety of combinations of vehicle, positional energy savings, and road segment, characteristics.

The contributions of this work are as follows. 
\begin{enumerate}[itemsep=1pt]
     
\item We provide the first comprehensive classification of the solvability of vehicle sequencing and resequencing problems within vehicle platoons.

\item For many of these problems, we describe the most efficient optimal algorithm so far developed.

\item For two intractable problems, we describe approximation schemes that can be used to find solutions that are as close to optimal as needed.

\item For other intractable problems, we describe simple heuristic methods that can easily be implemented by truck platoon companies, and analyze their worst case performance.

\item 
The performance of the heuristics is also evaluated by a computational study, which estimates that the typical performance of the heuristics is routinely less than 1

\item We analyze the performance of a simplified approach to scheduling and rescheduling that only considers the energy saving rates among the first few vehicles.

\item We provide an efficient optimal algorithm for the situation where vehicle position changes are limited by safety concerns during on-road resequencing. 

\end{enumerate}
 
The remainder of this paper is organized as follows. Section~\ref{sec:literature} reviews the literature on sequencing and resequencing in vehicle platoons. Section~\ref{sec:preliminaries} describes our notation, the problem environment, and the objectives considered. In Section~\ref{sec-min-total}, we consider the objective of minimizing total energy usage. Since low SoC levels are a significant problem in practice, in Section~\ref{sec:max-min} we consider the objective of balancing vehicle usage by minimizing the maximum energy usage of any vehicle. Section~\ref{sec:extensions} considers two extensions of the problem. Finally, Section~\ref{sec:conclude} presents a conclusion and some suggestions for further research.

\section{Related Literature}  \label{sec:literature}

The sequencing and resequencing of vehicles within a platoon directly determine the distribution of aerodynamic benefits and therefore strongly influence total energy consumption, battery depletion, fairness among vehicles, and ultimately the economic viability of platooning. Despite this importance, our review of the literature reveals four major limitations. First, most existing approaches are evaluated primarily through simulation rather than theoretical analysis. Second, computational complexity, approximability, and structural properties of the underlying optimization problems are largely unknown. Third, exact algorithms have received little attention and have been applied only to small instances. Fourth, when heuristic or reinforcement-learning approaches are used, worst-case performance guarantees are generally absent. To our knowledge, the vehicle platoon sequencing and resequencing problems remain among the least theoretically understood platooning problems despite their direct impact on energy savings.

\subsection{Surveys}  \label{subsec:litsurveys}

Because of its evident practical importance and potential benefits, the vehicle platooning problem is attracting a rapidly growing literature. We begin by discussing five surveys of the vehicle platooning literature. Bhoopalam et al. (2018) discuss various planning problems that arise within vehicle platooning, along with several related operations research models. Four main problems are identified: (a) which trucks to form together into a platoon, (b) where and when to form the platoon, (c) the route to follow, and (d) how to sequence the vehicles within the platoon. Zhang et al. (2020) describe factors that contribute to fuel consumption, including vehicle speed and position within the platoon. They also identify directions for future research. Rebelo et al. (2024) provide a detailed literature review with a focus on environmental impacts. They identify a close link between lower energy consumption and a reduction in air pollution emissions. They discuss several field experiments that document platoon performance with different characteristics, and identify the need for additional research to support optimal platoon formation. Li et al. (2022) review the literature on platoon splitting and merging. They propose a three-step framework consisting of protocol design, trajectory planning, and vehicle control to unify the related studies. For each step, future research directions are proposed. Braiteh et al. (2025) provide an extensive survey of research work and software packages for vehicle platooning, with a particular focus on vehicle sequencing and resequencing, which is the topic of our work.
The issues highlighted by these surveys provide important motivation for the present study.

\subsection{Formation, routing, and coordination}  \label{subsec:litFRC}

The literature of vehicle platooning studies several broader general issues, including vehicle formation and routing, the integration of charging operations, and the role of cooperative behavior in platoons shared between several vehicle owners. Larson et al. (2015) study the use of a distributed system of controllers that coordinate the speed of vehicles in order to form platoons. In a simulation study of the German Autobahn system, they demonstrate savings of over 5\% by coordinating thousands of vehicles into many platoons. Larsson et al. (2015) consider a single platoon routing problem with homogeneous vehicles where fuel savings may be increased by deviating from a shortest route. They show that this problem is strongly {\em NP}-hard, and develop heuristics for problems with up to 200 vehicles. van de Hoef et al. (2018) prove that the platoon formation problem is strongly {\em NP}-hard, and develop a heuristic which enables significant fuel savings in a simulation study. When a platoon includes vehicles with multiple owners, the fact that energy savings are different at different positions in the platoon may lead to instability of platoon cooperation.
Sun and Yin (2019) formulate a cooperative game to establish a stable mechanism for distributing those savings between the vehicle owners. Liu et al. (2026) describe a procedure to coordinate en-route charging and vehicle resequencing decisions, and demonstrate its success against benchmark policies. Heterogeneity in charging power and energy consumption rate are both shown to reduce energy reduction, which complicates the formation of multiple platoons.

Collectively, these studies demonstrate that platoon formation, routing, charging coordination, and cooperation mechanisms are computationally challenging and often require heuristic approaches. However, sequencing decisions within an already formed platoon have received substantially less attention, despite their direct influence on energy consumption and fairness.
\subsection{Sequencing and resequencing}  \label{subsec:litsesreseq}

The focus of our work is on the sequencing and resequencing problem with heterogeneous trucks and a single vehicle owner.
Srisomboon and Lee (2021) describe four heuristic vehicle resequencing rules based on remaining SoC values. Based on a simulation study, they recommend a minimum fuel saving before resequencing. The recommended resequencing policy yields approximately 3\% energy reduction in two-vehicle platoons and 6\% savings in six-vehicle platoons.

Guo and Meng (2025) shift attention from energy minimization to fairness by minimizing the variance of vehicle SoC levels. They compare various heuristic rules using real-world route experiments: a brute-force enumeration algorithm, a fixed platoon with only one resequence used, resequencing by smallest SoC, and a local search procedure which swaps the positions of the vehicles with maximum and minimum SoC. Taking into account both heuristic performance and computation time, the latter two rules provide the best balance between solution quality and computational effort.

To address the combinatorial explosion of resequencing decisions, Zheng and Guo (2025) apply Bootstrapped Deep Q-Networks to the platoon resequencing problem. They model the problem with three objectives: minimization of final SoC variance, maximization of the minimum final SoC of any vehicle, and minimization of the number of vehicle position changes. Recognizing that the factorial growth of permutation-based actions becomes excessive for six or more vehicles, they apply their methodology to a platoon with five vehicles and a 197-mile route between Suzhou and Nanjing with five resequencing points. Their computational study shows significant improvements over heuristic rules for all three objectives, although at substantially higher computational cost.

Focusing on minimization of energy variance, Zheng et al. (2025) evaluate five deep reinforcement learning models in a simulation study calibrated with real highway data. The platoon travels along a 253-mile route between Chuzhou and Yancheng, and includes five predetermined points at which vehicle resequencing is possible. The results of this study recommend a Noisy Dueling DQN approach that provides robust performance. Peng et al. (2026) propose a congestion-aware 
platoon resequencing optimization framework for electric vehicles using deep reinforcement learning. Their experimental results show that, compared to existing reinforcement learning methods, the proposed framework reduces the frequency of platoon resequencing by 34.4\% and achieves a 23.6\% reduction in the final standard deviation of the SoC across all vehicles.

Existing studies evaluate their heuristics and reinforcement learning approaches computationally only without analyzing their theoretical performance bounds. In contrast, our work studies the underlying combinatorial optimization structure of sequencing and resequencing
decisions, proposes exact algorithms, and analyzes the worst-case performance of the approximation schemes and heuristic algorithms.

\subsection{Research gap and contributions}  \label{subsec:RGC}

Despite extensive research on platoon formation, routing, charging, and control, the theoretical foundations of sequencing and resequencing remain largely undeveloped.  Consequently, fundamental questions regarding tractability, approximability, and algorithm design remain largely unresolved.
Most resequencing papers optimize only locally and rely on simulation to validate their results. Few consider the natural objective of 
total energy minimization, or worst-case vehicle energy consumption. 
Further, within the recent work that uses  reinforcement learning, solution quality guarantees are absent, optimality is unknown, and computational effort increases rapidly.
Vehicle platoon sequencing shares characteristics with machine scheduling, sequencing (Pinedo 2022), and load-balancing problems Kellerer et al. 2010), including the assignment of heterogeneous entities to ordered positions and the balancing of cumulative costs across agents.
However, the energy-sharing structure of platoons introduces new dependencies that prevent direct application of existing results. Most importantly, sequencing and resequencing are not merely control decisions but fundamental combinatorial optimization problems. Understanding their computational complexity is therefore a prerequisite for the design of scalable exact, approximation, and learning-based methods. Moreover, we develop heuristics with worst-case performance guarantees for several intractable problems and support these results with a computational study. As platoons become larger and journeys longer, the value of such theoretically grounded methods is likely to increase.

As mentioned above, our study of the above literature identifies two  commonly considered objectives. The first is the natural objective of minimizing total energy usage. The second, motivated by the need to balance load between vehicles and avoid low SoC occurrences, is minimization of the maximum energy usage of any vehicle. The remainder of the paper provides the first systematic complexity and algorithmic study of these two problems, which are studied in detail in Sections~\ref{sec-min-total} and \ref{sec:max-min} below, respectively.

\section{Preliminaries} \label{sec:preliminaries}

Section~\ref{subsec:problem} describes the problems studied. Section~\ref{subsec:prelimres} provides  preliminary results that are used at several points in the paper. Section~\ref{subsec:overview} provides an overview of the results in the paper.

\subsection{Problem statement and notation} \label{subsec:problem}

We study the following resequencing problem involving a set of $n$ cooperatively platooning electric vehicles (EVs), denoted by $V=\{1, \ldots, n\}$. These EVs travel together through $m$ road segments, denoted by $1, 2, \ldots, m$. The starting and ending points of segment $i$ are denoted as locations $i$ and $i+1$, respectively, for $i=1, \ldots, m$. Locations 1 and $m+1$ are the common origin (starting point) and common destination (ending point) of all the vehicles, respectively. The travel distance over each segment $i$ is known and denoted by $D_{i}$, for $i=1, \ldots, m$, where we let $D_{\max} = \max_{1 \le i \le m} \{D_i\}$. 

Each vehicle $j$ is characterized by four parameters: (i) battery capacity, denoted by $C_j$, which is the maximum amount of energy (i.e., electricity) its battery can store, (ii) usage rate, denoted by $\delta_j$, which is the amount of energy consumption per mile, (iii) an initial state of charge at the origin, denoted by $S_{j1}$, where the SoC of a vehicle is a value between 0\% and 100\%, and defined as the percentage of the remaining energy relative to its capacity, and (iv) a minimum allowable SoC denoted by $\beta_j$, below which the vehicle would become undrivable. To avoid obvious infeasibility, it is assumed that $\beta_j<S_{j1}$ for each vehicle $j$. We let $\delta_{\max} =\max_{1 \le j \le n} \{\delta_j\}$ and $\delta_{\min} =\min_{1 \le j \le n} \{\delta_j\}$. Different vehicles may have different or identical parameters (i) through (iv). Observe that the maximum driving range of a vehicle $j$ that does not benefit from energy savings as a result of platooning is $[C_j(S_{j1} - \beta_j)/\delta_j]$. Let $Q_{k1} = S_{k1}C_k$, for $1\le k\le n$, which are the initial energy levels of the vehicles, and $Q_{\max} = \max_{1 \le k \le n} \{ S_{k1}C_k\}$.

The vehicles may be resequenced, at each location $i$, for $i=1,\ldots, m$. They then  travel in the resequenced platoon over the road segment from location $i$ to the next location $i+1$, where they are possibly resequenced again. We denote the state of charge of vehicle $j$ after reaching location $i$ by $S_{ji}$.
If a vehicle $j$ travels alone, it  consumes an average of $\delta_j$ units of energy from its battery per unit distance. Consequently, if a vehicle $j$ travels alone on its own from location $i$ to location $i+1$, it consumes $D_i\delta_j$ units of energy, and as a result, its SoC  decreases by 
$\frac{D_{i}\delta_j}{C_j}$. 

However, if $n$ vehicles travel in a platoon, they can save some energy, and the savings for a vehicle depend on its position  in the platoon. If vehicle $j$ is in the $k$th position of the platoon, then the average energy usage per unit distance becomes $\delta_j(1-\eta_{k})$, where $\eta_{k}$ is the usage rate reduction associated with  the $k$th position, and $0\leq \eta_{k}<1$. In general, a later position yields more energy savings than an earlier position, i.e.,  $\eta_{k+1} \geq \eta_{k}$, for $k=1, \ldots, n-1$.  
Now, suppose that vehicle $j$ is at position $j_i$ after being resequenced at location $i$, for $i=1, \ldots, m$, then the total energy consumed by the vehicle by the time it reaches location $i+1$, denoted as $U_{j,i+1}$ and its SoC at location $i+1$ are, respectively, 
\vspace{-0.1in}
\begin{eqnarray} 
U_{j,i+1} & = & U_{ji} + D_i\delta_j(1-\eta_{j_i}) = \delta_j\sum_{h=1}^i D_h(1-\eta_{j_h}), \quad \mbox{and} \label{Usage-def-i}\\
S_{j,i+1} & = & S_{ji} - \frac{D_{i}\delta_j}{C_j}(1-\eta_{j_{i}}) =  S_{j1} - \frac{\delta_j}{C_j}\left[\sum_{h=1}^{i}D_h(1-\eta_{j_{h}})\right].  \label{SoC-def-i}
\end{eqnarray}
 As a result, the total energy usage by vehicle $j$ and the final SoC of vehicle $j$ after reaching the destination location $m+1$, are, respectively, 
 \vspace{-0.1in}
\begin{eqnarray} 
U_{j,m+1} & = & \delta_j\sum_{i=1}^m D_i(1-\eta_{j_i}) = \delta_j\sum_{i=1}^m D_i - \delta_j\sum_{i=1}^m D_i\eta_{j_i}, \quad \mbox{and} \label{Usage-def}\\
S_{j,m+1} & = & S_{j1} - \frac{\delta_j\left[\sum_{i=1}^{m}D_i\right]}{C_j} + \frac{\delta_j \left[\sum_{i=1}^{m}(D_{i}\eta_{j_{i}})\right]}{C_j}. \label{SoC-def} 
\end{eqnarray}

The Vehicle Platooning problem finds a sequence of the vehicles at each location (i.e., a solution that resequences the platoon at each location) such that a certain objective function is optimized. We consider two specific problems: 
\begin{itemize}[itemsep=1pt]
\item Problem~1: Minimize the total energy usage of the vehicles, i.e.,  $\sum_{j=1}^n U_{j,m+1}$, subject to the constraint that each vehicle's final SoC is at least at a minimum required level, i.e., $S_{j,m+1}\ge \beta_j$, for $j=1, \ldots, n$, where $0\leq \beta_j< S_{j1}$ is the lowest allowable SoC for vehicle $j$.
\item Problem~2: Minimize the maximum energy usage  of the vehicles, i.e., $\max\{U_{j,m+1}~|~j=1, \ldots, n\}$, subject to the same constraint as above.
\end{itemize}

We observe that the objective of Problem~1 is equivalent to maximizing the total usage savings of the vehicles that results from platooning, i.e., maximizing $\sum_{j=1}^n \delta_j \left[\sum_{i=1}^{m}(D_{i}\eta_{j_{i}})\right]$. Further, the  constraint on both these problems ensures that the total energy usage by each vehicle is no more than a given threshold, i.e., $U_{j,m+1} \le A_j$, where $A_j=(S_{j1}-\beta_j)C_j$ is the allowable energy usage by the vehicle, for $j=1, \ldots, n.$ 

To provide clear definitions for the variety of vehicle platooning problems discussed in our work, we introduce the following
three field notation $\alpha \; | \; \beta \; | \; \gamma$.
In the first field, we use $n$ (respectively, $\bar n$) to denote that the number of vehicles is arbitrary, i.e. part of the instance input, (resp., fixed), and similarly for the number of road segments denoted by $m$ and $\bar m$. The second field contains special case assumptions and constraints commonly found in practice, for example $\delta_j = \delta$ specifies that all vehicles have the same energy usage rate, as typically occurs with a homogeneous fleet of vehicles in the platoon.
The third field specifies the objective being considered, i.e. Problem~1 or Problem~2 defined above. For conciseness, we use $``TU"$ for the minimization of total usage objective in Problem~1, and $``MU"$ for the minimization of the maximum usage in Problem~2. As an example of our three-field notation, problem~$n, \bar m\: | \: \beta_j = \beta \; | \; TU$ requires the minimization of total energy usage by a platoon with an arbitrary number of vehicles, operating over a fixed number of road segments, where each vehicle has the same minimum required final SoC. All the problems contain the constraint that the final SoC of each vehicle $j$ is at least $\beta_j$. For conciseness, this constraint is not added explicitly in the 3-field notation.

The purpose of this work is to (i) investigate the solvability of various vehicle platooning problems parametrized by 
$\alpha \; | \; \beta \; | \; \gamma$, in order to provide the most efficient algorithm possible, or where necessary a proof of intractability, (ii) provide fast heuristics for several intractable cases of the problems and analyze their worst-case performance, and (iii) provide a fully polynomial time approximation scheme (FPTAS) for two general cases of the problems.

\subsection{Preliminary results} 
\label{subsec:prelimres}

Some results for two given sequences of positive numbers are used multiple times below.

\begin{lemma} \label{lemma1}
Given two sequences of positive numbers $(x_1, \ldots, x_h)$ and $(y_1, \ldots, y_h)$, where $x_j$'s are in a non-increasing order, i.e., $x_1\ge \cdots \ge x_h$, and $y_j$'s are in a non-decreasing order, i.e., $y_1\le \cdots \le y_h$. For any permutation $([1], \ldots, [h])$ of $(1, \ldots, h)$, we have: 
\vspace{-0.1in}
\begin{eqnarray}
& & (i) \quad \max\{x_jy_j ~|~ j=1, \ldots, h\} \le  \max\{x_jy_{[j]} ~|~j=1, \ldots, h\}, \label{xy-result1} \\
& & (ii) \quad \max\{x_j+y_j ~|~ j=1, \ldots, h\} \le \max\{x_j+y_{[j]} ~|~j=1, \ldots, h\}, \label{xy-result2} \\
& & (iii) \quad \sum_{j=1}^h x_jy_j  \le  \sum_{j=1}^h x_jy_{[j]}. \label{xy-result3}
\end{eqnarray}
\end{lemma}
{\bf Proof:} We prove these results by showing that given any permutation of $(1, \ldots, h)$, denoted as $\pi=(\pi_1, \ldots, \pi_h)$,  we can iteratively convert it to $\pi_0=(1, \ldots, h)$ such that in each step of the conversion process, each of the three measures, namely, the maximum pairwise product, the maximum pairwise sum, and the sum of the pairwise products, remains the same or becomes smaller. Suppose that $\pi$  differs from $\pi_0$ starting from the $k$th element, i.e., $\pi_i=i$ for $i=1,\ldots, k-1$, but $\pi_k\ne k$, for some $k\ge 1$. Further suppose that $\pi_k=q$ and $\pi_u=k$. Clearly, $q>k$ and $u>k$. We generate a new permutation $\pi'$ from $\pi$ by replacing the $k$th element in $\pi$ by $k$ and the $u$th element in $\pi$ by $q$.  Thus, $\pi'_k=k<\pi_k$ and $\pi'_u=q>k$, and $\pi'_i=\pi_i$ for $i\in \{1, \ldots, h\}\setminus \{k, u\}$. This further implies $y_{\pi'_k}=y_k=y_{\pi_u}\le y_{\pi_k}$, $y_{\pi'_u}=y_q=y_{\pi_k}$, and $y_{\pi'_i}=y_{\pi_i}$
for $i\in \{1, \ldots, h\}\setminus \{k, u\}$. This, together with the fact that $x_k\ge x_u$, gives
\vspace{-0.2in}
\begin{eqnarray*}
\max\{x_ky_{\pi'_k}, x_uy_{\pi'_u}\} &\le& \max\{x_ky_{\pi_k}, x_uy_{\pi_k}\}
= x_ky_{\pi_k} \le \max\{x_ky_{\pi_k}, x_uy_{\pi_u}\}, \\
\max\{x_k+y_{\pi'_k}, x_u+y_{\pi'_u}\} &\le& \max\{x_k+y_{\pi_k}, x_u+y_{\pi_k}\}
= x_k+y_{\pi_k} \le \max\{x_k+y_{\pi_k}, x_u+y_{\pi_u}\}, \\
x_ky_{\pi'_k}+x_uy_{\pi'_u} &=& x_ky_{\pi_u}+x_uy_{\pi_k} \\
  & = &  x_ky_{\pi_k} + x_uy_{\pi_u} - (x_k-x_u)(y_{\pi_k}-y_{\pi_u})  \le  x_ky_{\pi_k} + x_uy_{\pi_u}.
\end{eqnarray*}
These results imply that
\vspace{-0.2in}
\begin{eqnarray*}
\max\{x_jy_{\pi'_j} ~|~ j=1, \ldots, h\} &=& \max\left\{x_ky_{\pi'_k}, x_uy_{\pi'_u}, \max\{x_jy_{\pi'_j} ~|~ j\in\{1,\ldots,h\}\setminus\{k,u\}\right\} \\
&\le & \max\left\{x_ky_{\pi_k}, x_uy_{\pi_u}, \max\{x_jy_{\pi_j} ~|~j\in \{1,\ldots,h\}\setminus\{k,u\}\right\}\\
&=& \max\{x_jy_{\pi_j} ~|~ j=1, \ldots, h\} 
\end{eqnarray*}
Similarly, we can show that
\vspace{-0.2in}
\begin{eqnarray*}
\max\{x_j+y_{\pi'_j} ~|~ j=1, \ldots, h\} & \le &
\max\{x_j+y_{\pi_j} ~|~ j=1, \ldots, h\}, \quad \mbox{and} \\
\sum_{j=1}^h x_jy_{\pi'_j}  &\le & \sum_{j=1}^h x_jy_{\pi_j}. 
\end{eqnarray*}
Hence, in going from $\pi$ to $\pi'$, each of the three measures remains the same or becomes smaller. Continuing this, we eventually arrive at permutation $\pi_0$ without increasing any of the three measures. \blot

\subsection{Overview of the results} \label{subsec:overview}

Table~\ref{tab:results} provides an overview of the results in the paper. 

\begin{table}
\tiny
\centering
\caption{Overview of Algorithm, Complexity, and Heuristic Results for Vehicle Platooning.}
\begin{tabular}{||c|c|c||c|c|c||} \hline
$\alpha$ & $\beta$ & $\gamma$ & Result &
Time / Worst-Case Bound & Reference \\ \hline
&&&&& \\[-9pt] \hline
$n,1$ & & $TU$ & Optimal & $O(n^2)$ & Proposition~\ref{prop2.1} \\
$n,1$ & $\delta_j = \delta$ & $TU$ & Optimal & $O(n \log n)$ & Proposition~\ref{prop:3.1-Rule1} \\
$n,1$ & $A_j = A$ & $TU$ & Optimal & $O(n \log n)$ & Proposition~\ref{prop:3.1-Rule2} \\
$n,2$ & $C_j = C, \delta_j = \delta, S_{j1} = S, \beta_j = \beta$ & $TU$ & Optimal & $O(n)$ & Proposition~\ref{prop3} \\
$n,2$ & $C_j = C, \delta_j = \delta$ & $TU$ & Strongly {\em NP}-hard & & Proposition~\ref{prop2.2.2} \\
$n,2$ & $C_j = C, S_{j1} = S, \beta_j = \beta$ & $TU$ & Strongly {\em NP}-hard & & Proposition~\ref{prop2.2.3} \\
$n,2$ & $\delta_j = \delta, S_{j1} = S, \beta_j = \beta$ & $TU$ & Strongly {\em NP}-hard & & Proposition~\ref{prop2.2.3} \\
$n,\bar{m}\ge 3$  & $C_j=C, \delta_j = \delta, S_{j1} = S, \beta_j = \beta$ & $TU$ & Open & & Section~\ref{sec4-openprob} \\ 
$\bar n \! \ge \! 2,m$ & $C_j = C, \delta_j = \delta, S_{j1} = S, \beta_j = \beta$ & $TU$ & Ordinarily {\em NP}-hard & $O(mnn!Q^n_{\max})$ & Propositions~\ref{prop:2.3bnpc},\ref{prop:FJ} \\
$2,m$ & $C_j = C, \delta_j = \delta, S_{j1} = S, \beta_j = \beta$ & $TU$ & FPTAS & $O(m [(1 \!\! + \!\! 2m/\epsilon) 
\ln{(m D_{\max} \delta_{\max})}]^n n!)$ & Proposition~\ref{DP1fptas} \\
$n,m$ & $C_j = C, \delta_j = \delta, S_{j1} = S, \beta_j = \beta$ & $TU$ & Strongly {\em NP}-hard & & Proposition~\ref{prop:2.5-snpc} \\
$n,m$ & $S_{j1}C_j \ge \bar L$ & $TU$ & Optimal & $O(n \log n)$ & Proposition~\ref{prop-rule4} \\ \hline
$n,1$ & & $MU$ & Optimal & $O(n^2)$ & Proposition~\ref{prop:3.1-MaxMatch} \\
$n,2$ & $C_j = C, \delta_j = \delta, S_{j1} = S, \beta_j = \beta$ & $MU$ & Optimal & $O(n)$ & Proposition~\ref{prop:3.2-Rule1} \\
$n,2$ & $C_j = C, \delta_j = \delta, S_{j1} = S$ & $MU$ & Strongly {\em NP}-hard & & Corollary~\ref{prop:3.2-SNP1} \\
$n,2$ & $C_j = C, \delta_j = \delta, \beta_j = \beta$ & $MU$ & Strongly {\em NP}-hard & & Corollary~\ref{prop:3.2-SNP1} \\
$n,2$ & $C_j = C, S_{j1} = S, \beta_j = \beta$ & $MU$ & Strongly {\em NP}-hard & & Corollary~\ref{prop:3.2-SNP2} \\
$n,2$ & $\delta_j = \delta, S_{j1} = S, \beta_j = \beta$ & $MU$ & Strongly {\em NP}-hard & & Corollary~\ref{prop:3.2-SNP2} \\
$n,2$ & $S_{j1}C_j \ge \bar L, \delta_j = \delta$ & $MU$ & Optimal & $O(n)$ & Proposition~\ref{prop:3.2-easy} \\
$n,2$ & $C_j=C, S_{j1} = S, \beta_j = \beta, SC \ge \bar L$ & $MU$ & Strongly {\em NP}-hard & & Proposition~\ref{prop:3.2-SNP3} \\
$n,\bar{m}\ge 3$ & $C_j=C, \delta_j = \delta, S_{j1} = S, \beta_j = \beta$ & $MU$ & Open & & Section~\ref{sec5-openprob} \\
$\bar n \! \ge \! 2,m$ & $C_j=C, \delta_j = \delta, S_{j1} = S, \beta_j = \beta$ & $MU$ & Ordinarily {\em NP}-hard & & Proposition~\ref{prop:3.4-ONP} \\
$\bar n, m$ & & $MU$ & Ordinarily {\em NP}-hard & $O(mn!Q_{\max}^n)$ & Proposition~\ref{prop:3.4-Enumerate} \\
$\bar n, m$ & & $MU$ & FPTAS & $O(m[(1 + 2m/\epsilon) \ln(Q_{\max})]^n n!)$ & Corollary~\ref{cor:ESfptas} \\
$\bar n, m$ & $C_j = C, \delta_j = \delta, S_{j1} = S, \beta_j = \beta$ & $MU$ & Strongly {\em NP}-hard & & Proposition~\ref{prop:3.5snpc} \\
$n, m$ & $C_j \! = \! C, \delta_j \! = \! \delta, S_{j1} \! = \! S \! \ge \! \bar L/C, \beta_j \! = \! \beta$ & $MU$ & Strongly {\em NP}-hard & & Proposition~\ref{prop:3.5snpc} \\
$n, m$ & $S_{j1}C_j \ge \bar L, \delta_j = \delta$ & $MU$ & Strongly {\em NP}-hard,  Heuristic & $1 + (\eta_n - \eta_1)/(1 - \eta_1)$ & Proposition~\ref{prop:3.5-heuristic} \\
$n, m$ & $S_{j1}C_j \ge \bar L$ & $MU$ & Strongly {\em NP}-hard, Heuristic & $\frac{\max_{1 \le j \le n} \{\delta_j(1 - \eta_j)\}}{\min_{1 \le j \le n} \{\delta_j(1 - \eta_j)\}}$ & Proposition~\ref{prop:3.5-heu2} \\
$n, m$ & $\pi_1, pc \le K$ & $MU$ & Optimal & $O((\log n \! \! + \!\! m \log K)n(2K)^{(2m^2 K + m)})$ & Proposition~\ref{prop:swap} \\ \hline
$n, \bar m$ & $(\eta_1,\ldots,\eta_K)$ & $TU, \! MU $ & Optimal & $O(mn^{Km+1})$ & Proposition~\ref{prop-6.2-1} \\
$n, \bar m$ & 
& $TU, \! MU $ & Heuristic & $1 + (\eta_n - \eta_K)/(1 - \eta_n)$ & Proposition~\ref{prop-6.2-2} \\ \hline
\end{tabular}
{\footnotesize \singlespacing \raggedright \textit{Note:} The following assumptions and constraints that appear in the $\beta$ field in the table are defined precisely below. \\``$S_{j1}C_j\ge \bar{L}$", ``$SC\ge \bar{L}"$ and ``$S\ge \bar{L}/C$" all represent the assumption that each vehicle's initial level of energy is sufficiently large that the constraint on its final SoC is always satisfied in any solution.\\ ``$\pi_1, pc\le K$" represents the constraint that at most $K$ position changes are allowed from an initial platoon sequence $\pi_1$. \\``$(\eta_1, \cdots, \eta_K)$" represents the case where there is some $K$ such that $\eta_1\le \cdots \le \eta_K=\eta_{K+1}=\cdots=\eta_{n}$. \par}
\label{tab:results}
\setlength{\unitlength}{1pt}
\end{table}

\section{Minimizing Total Energy Usage} \label{sec-min-total}

In this section, we consider various cases of problem $\alpha \; | \; \beta \; |TU$ by either providing an efficient algorithm or proving that it is {\em NP}-hard for each case. For ease of presentation, we define $\lambda_k=1-\eta_k$, for $k=1, \ldots, n$. Thus, $\lambda_1 \ge \lambda_2 \ge \cdots \ge \lambda_n$, and $\lambda_{k}$ can be viewed as the realized percentage of the usage rate for a vehicle at position $k$. Given a solution, i.e., specific positions occupied by each  vehicle $j$ over the $m$ road segments, denoted as $(j_1, \ldots, j_m)$, we can rewrite the total energy usage of the vehicles as  $\sum_{j=1}^n \left(\delta_j\sum_{i=1}^{m}D_i\lambda_{j_i}\right)$. 
We can also rewrite the minimum allowable SoC  constraint of the problem as: $S_{j1} - \frac{\delta_j}{C_j}\sum_{i=1}^mD_{i}\lambda_{j_i} \ge \beta_j$, for $j=1, \ldots, n$.

\subsection{Problem with $m=1$ and Arbitrary $n$} \label{sec2.1}

When there is only a single road segment, i.e., the vehicles travel from location 1 as their origin  to location 2 as their destination, the problem is to find a vehicle sequence for the platoon at location 1, i.e., a permutation of the $n$ positions, $(1, 2, \ldots, n)$, denoted as $([1], [2], \ldots, [n])$ such that the total energy usage by the vehicles for traveling from location 1 to location 2, i.e., $\sum_{j=1}^n \left(\delta_jD_1\lambda_{[j]}\right)$, is minimized subject to the constraint that $S_{j1} - \frac{\delta_j}{C_j}D_1\lambda_{[j]} \ge \beta_j$, for $j=1, \ldots, n$. 

We propose a simple algorithm to solve this problem.
The idea of this algorithm is that it is always optimal to assign a feasible vehicle with a larger usage rate to an earlier position. 

\medskip

\noindent {\bf Algorithm Delta} \\
\noindent {\bf Step 0:} Reindex the vehicles in non-decreasing order of $\delta_j$ (and in  increasing order of the initial vehicle indices in case of identical $\delta_j$'s). Define $E_k$, for $k=1, \ldots, n$, to be the subset of feasible vehicles for position $k$, i.e., $E_k=\{j~|~ S_{j1} - \frac{\delta_j}{C_j}D_1\lambda_{k} \ge \beta_j, j=1, \ldots, n\}$. Clearly, $E_1\subseteq E_2 \cdots \subseteq E_n$. Set position counter $q=1$ and define the set of scheduled vehicles $V=\phi$.\\
\noindent {\bf Step 1:} If $E_q$ is empty, then the problem is infeasible. Otherwise, assign the vehicle in $E_q$ with the smallest index, denoted as vehicle $u$, to position $q$. \\
\noindent {\bf Step 2:} If $q=n$, stop. Otherwise, update $V = V\cup \{u\}$,  $E_{q+1} = E_{q+1}\setminus V$, and $q=q+1$. Go to Step 1. 

\begin{proposition} \label{prop2.1}
Algorithm Delta solves problem $n,1~|~|~TU$ in $O(n^2)$ time by either proving the problem to be infeasible or finding an optimal solution when the problem is feasible.  
\end{proposition}
{\bf Proof:} 
We prove the optimality of the algorithm by showing that any feasible solution $\pi$ which differs from the solution  generated by this algorithm can be converted into the one generated by the algorithm without increasing the total energy usage. Suppose that in $\pi$, the vehicles assigned to the first $p$ positions, for some $0\leq p\leq n-2$, are the same as in the solution  generated by this algorithm, but the vehicle assigned to position $p+1$, denoted as vehicle $j_{p+1}(\pi)$, is different from that in the solution generated by the algorithm, which is denoted as vehicle $j^*_{p+1}$. 
It can be seen that both vehicles $j_{p+1}(\pi)$ and $j^*_{p+1}$ must be in $E_{p+1}$ updated in Step 2 when $q=p$. The set of vehicles in $E_{p+1}$ immediately before the algorithm goes back to Step 1 is the set of eligible vehicles for position $p+1$. Thus, by definition, $j^*_{p+1} < j_{p+1}(\pi)$, and hence $\delta_{j^*_{p+1}}\leq \delta_{j_{p+1}(\pi)}$. Suppose that in $\pi$, vehicle $j^*_{p+1}$ is assigned to position $k$, where $k$ must be greater than $p+1$. Hence, $\lambda_{p+1}\geq \lambda_{k}$. 
Now, we modify solution $\pi$ by swapping the positions of vehicles $j_{p+1}(\pi)$ and $j^*_{p+1}$, i.e., assign $j^*_{p+1}$ to position $p+1$ and $j_{p+1}(\pi)$ to position $k$. The resulting solution is still feasible because $j^*_{p+1}\in E_{p+1}$ and $j_{p+1}(\pi)\in E_{p+1}\subseteq E_{k}$. The total energy consumption in the modified solution is reduced by
\vspace{-0.1in}
\begin{eqnarray}
& & D_1[\delta_{j_{p+1}(\pi)}\lambda_{p+1} + \delta_{j^*_{p+1}}\lambda_{k} - (\delta_{j_{p+1}(\pi)}\lambda_{k} + \delta_{j^*_{p+1}}\lambda_{p+1})] \nonumber \\
& & = D_1(\delta_{j_{p+1}(\pi)}-\delta_{j^*_{p+1}})(\lambda_{p+1} - \lambda_{k}) \geq 0. \label{eq-delta}
\end{eqnarray}
We repeat the above procedure to modify $\pi$ until it becomes exactly the same as the solution generated by the algorithm. Step~0 of the algorithm requires $O(n \log n)$ time. 
The other two steps run a total of  $n$ iterations, and in each iteration, it takes at most $O(n)$ time to update $V$ and $E_{q+1}$. Thus, the total time required by the algorithm is $O(n^2)$.  \blot

\medskip

Next, we show that there is a more efficient algorithm which solves the problem when (i) the energy usage rates of the vehicles are identical, i.e., $\delta_j\equiv \delta$, or (ii) when the allowable usages of the vehicles are identical, i.e.,  $A_j\equiv A$, where $A_j=C_j(S_{j1}-\beta_j)$. 
In Case (i), if the problem is feasible, then the total usage of the vehicles in any solution is a constant, $\sum_{j=1}^n D_1\delta(1-\eta_j)$, independent of how the vehicles are sequenced. Thus, the problem reduces to finding a feasible solution if it exists. The following greedy rule is optimal.
\bigskip

\noindent {\bf Rule 1:} Reindex the vehicles in non-increasing order of  $A_j$ values, and assign vehicle $j$ to the $j$th position, for $j=1, \ldots, n$. 

\begin{proposition}  \label{prop:3.1-Rule1}
Rule 1 solves problem $n,1~|~\delta_j=\delta~|~TU$ in $O(n\log n)$ time by either showing the infeasibility of the problem or finding an optimal solution if the problem is feasible. 
\end{proposition}
{\bf Proof:} The quantity $A_j=C_j(S_{j1}-\beta_j)$ is the allowed amount of energy vehicle $j$ can use without violating the maximum usage constraint, for $j=1, \ldots, n$. Thus, it is optimal to assign the vehicles in non-increasing order of $A_j$ to the positions in non-increasing order of the actual usage rates $\lambda_j$. If the solution found is infeasible, then the problem is infeasible. The main computational requirement is sorting the vehicles.  \blot

When the allowable usages of the vehicles $A_j$ are identical (which includes the case where $C_j=C$, $S_{j1} = S_1$ and $\beta_j = \beta$), then it yields a feasible solution, if one exists, when a vehicle with a larger $\delta_j$ is assigned to a later position. In fact, this rule is also optimal if the problem is feasible. This rule is formally stated as follows.  

\bigskip

\noindent {\bf Rule 2:} Reindex the vehicles in non-decreasing order of  $\delta_j$ values, and assign vehicle $j$ to the $j$th position, for $j=1, \ldots, n$. 

\begin{proposition}  \label{prop:3.1-Rule2}
Rule 2 solves problem $n,1~|~A_j=A~|~TU$ in $O(n\log n)$ time by either showing the infeasibility of the problem or finding an optimal solution if one exists. 
\end{proposition}
{\bf Proof:} First we show that if the problem is feasible, then the solution generated by this rule is feasible. If vehicle $j$ is assigned to position $[j]$, then the energy usage by the vehicle is $D_1\delta_j\lambda_{[j]}$, for $j=1,\ldots,n$. By Lemma~\ref{lemma1}, part (i), we can see that Rule 2 minimizes the maximum energy usage of the vehicles. Since the available energy of each vehicle is the same, this rule generates a feasible solution if the problem is feasible.

Again by Lemma~\ref{lemma1}, part (iii), the solution generated by this rule has the least total energy usage, and thus, it is also optimal if it is feasible. \blot

\subsection{Problem with $m=2$ and Arbitrary $n$} \label{sec2.2}

When there are two road segments, i.e., from the origin location 1 to location 2, and from location 2 to location 3, where location 3 is the final destination, the problem is to minimize $\sum_{j=1}^n \left(\delta_j\sum_{i=1}^{2}D_i\lambda_{j_i}\right)$, subject to the constraints: $S_{j1} - \frac{\delta_j}{C_j}\sum_{i=1}^2D_{i}\lambda_{j_i} \ge \beta_j$, for $j=1, \ldots, n$. 
We consider several different cases as follows:

Case 1: identical vehicles (i.e., $C_j\equiv C$ and $\delta_j\equiv \delta$), identical initial SoCs (i.e., $S_{j1}\equiv S$), and identical threshold on the vehicles final SoCs (i.e., $\beta_j\equiv \beta$). 

Case 2: identical vehicles (i.e., $C_j\equiv C$ and $\delta_j\equiv \delta$), and non-identical initial SoCs (i.e., $S_{j1}\not\equiv S$) or non-identical threshold on the vehicles final SoCs (i.e., $\beta_j\not\equiv \beta$). 

Case 3: non-identical vehicles (i.e., $C_j\not\equiv C$ or $\delta_j\not\equiv \delta$), identical initial SoCs (i.e., $S_{j1} \equiv S$), and identical threshold on the vehicles final SoCs (i.e., $\beta_j\equiv \beta$). 

\subsubsection{Problem $n,2~|~C_j= C, \delta_j= \delta, S_{j1}= S,\beta_j= \beta~|~TU$}

It can be seen that in this problem, any feasible solution has the same total energy usage and hence is optimal. Thus, we just need to find a feasible solution. Since the vehicles are identical, to find a feasible solution, we only need to find a feasible solution to the following perfect matching problem: Given $n$ positions $(1, \ldots, n)$ at location 1, and $n$ positions $(1, \ldots, n)$ at location 2. Construct a graph $G$ consisting of $2n$ vertices representing the $n$ positions at location 1 and the $n$ positions at location 2. For any  position $i$ at location 1 and any position $j$ at location 2, create an edge $(i,j)$ if when a vehicle is assigned to this edge (i.e., the vehicle is assigned to position $i$ at location 1 and position $j$ at location 2), its final SoC meets the threshold constraint, i.e.,  $S - \frac{\delta}{C}(D_1\lambda_{i}+D_2\lambda_{j})  \ge \beta$.  Find a perfect matching in the constructed graph $G$ if it exists. The following result simplifies this search.

\begin{lemma} \label{lemma2}
A perfect matching exists in $G$ if and only if $G$ contains  edge $(i, n + 1 - i)$ for every $i = 1,\ldots,n$.
\end{lemma}
{\bf Proof:} $(\Rightarrow)$. If $G$ contains edge $(i, n + 1 - i)$ for every $i = 1,\ldots,n$, then these edges together form a perfect matching. 

\noindent $(\Leftarrow)$. If a perfect matching exists, we denote it as $M$, noting that it must contain exactly $n$ edges.  We first show that if $M$ contains two edges $(a, b)$ and $(c, d)$ with $a<c$ and $b<d$, then $G$ must contain edges $(a, d)$ and $(c, b)$. The existence of edges $(a, b)$ and $(c, d)$ implies that 
\vspace{-0.15in}
\begin{eqnarray}
    \frac{\delta}{C}(D_1\lambda_{a}+D_2\lambda_{b}) & \leq & S-\beta. \label{prop1-eq1} \\
    \frac{\delta}{C}(D_1\lambda_{c}+D_2\lambda_{d}) & \leq &  S-\beta. \label{prop1-eq2}
\end{eqnarray}
Since $\lambda_1 \ge \cdots \ge \lambda_n$, the fact that $a<c$ and $b<d$ implies that $\lambda_a\ge \lambda_c$ and $\lambda_b\ge \lambda_d$. Thus, (\ref{prop1-eq1}) and (\ref{prop1-eq2}) imply that
\vspace{-0.1in}
\begin{eqnarray}
    \frac{\delta}{C}(D_1\lambda_{a}+D_2\lambda_{d}) & \leq & S-\beta. \label{prop1-eq3} \\
    \frac{\delta}{C}(D_1\lambda_{c}+D_2\lambda_{b}) & \leq & S-\beta. \label{prop1-eq4}
\end{eqnarray}
Inequalities (\ref{prop1-eq3}) and (\ref{prop1-eq4}) imply that both edges $(a, d)$ and $(c, b)$ are contained in $G$. We revise $M$ by replacing edges $(a, b)$ and $(c, d)$ by $(a, d)$ and $(c, b)$. Clearly, the revised $M$ is still a perfect matching. We repeat the above procedure until there are no two edges in $M$, $(u_1, v_1)$ and $(u_2, v_2)$ such that $u_1<u_2$ and $v_1<v_2$. Hence, the $n$ edges in $M$ can be written as $(1, j_1), (2, j_2), \ldots, (n, j_n)$, where $j_1>j_2>\cdots > j_n$. Since $(j_1, \ldots, j_n)$ is a permutation of $(1, \ldots, n)$, we have: $j_i=n-i+1$ for $i=1, \ldots, n$. This implies that $M$ consists of edges $(i, n + 1 - i)$ for $i = 1,\ldots,n$. Thus, these edges must exist in $G$. \blot

\medskip

From Lemma~\ref{lemma2}, to solve problem $n, 2~|~C_j= C, \delta_j= \delta, S_{j1}= S,\beta_j= \beta~|~TU$, we only need to check, for each $i=1, \ldots, n$,  if it is feasible to assign a vehicle to position $i$ in the first segment and position $n+1-i$ in the second segment, i.e., if $\frac{\delta}{C}(D_1\lambda_{i}+D_2\lambda_{n+1-i})  \leq  S-\beta$. If this constraint is satisfied for  $i=1, \ldots, n$, then this vehicle-position assignment gives a feasible and also optimal solution. Otherwise, the problem is infeasible. Thus, we have the following simple rule for this problem.

\medskip

\noindent {\bf Rule 3:} Assign vehicle $j$ to position $j$ in segment 1 and to position $n-j+1$ in segment 2, for $j=1, \ldots, n$. 

\begin{proposition} \label{prop3}
    Rule 3 solves problem $n,2~|~C_j=C, \delta_j=\delta,S_{j1}=S, \beta_j=\beta~|~TU$ in $O(n)$ time either by finding an optimal solution or verifying that the problem is infeasible. 
\end{proposition}
{\bf Proof:} This follows immediately from Lemma~\ref{lemma2} and the discussion above. \blot

\subsubsection{Problems $n,2~|~C_j= C, \delta_j= \delta, S_{j1}=S_1~|~TU$ and $n,2~|~C_j= C, \delta_j= \delta, \beta_j=\beta~|~TU$} \label{sec2.2.2}

Since the vehicles are identical in these problems, any feasible solution to these problems has the same objective value, $\sum_{j=1}^n \left(\delta\sum_{i=1}^{2}D_i\lambda_{j_i}\right) = \delta\sum_{j=1}^n(D_1\lambda_{j_1} +D_2\lambda_{j_2})$, where $(1_1, 2_1, \ldots, n_1)$ and $(1_2, 2_2, \ldots, n_2)$ are two permutations of $(1, 2, \ldots, n).$ 
Thus, we just need to find a feasible solution or verify that the problem is infeasible, with respect to the constraints: $S_{j1} - \frac{\delta}{C}\sum_{i=1}^2D_{i}\lambda_{j_i} \ge \beta_j$, for $j=1, \ldots, n$. However, doing so is not easy.  

\begin{proposition} \label{prop2.2.2}
Finding a feasible solution to the problems $n,2~|~C_j= C, \delta_j= \delta~|~TU$ is strongly NP-hard, even when $S_{j1} \equiv S_1$ or  $\beta_j \equiv \beta$. 
\end{proposition}
{\bf Proof:} We prove that finding a feasible solution to this problem is strongly NP-hard even when $S_{j1}\equiv S_1$ or $\beta_j\equiv \beta$,
by a reduction from the following restricted variant of Numerical 3-D Matching problem (RN3DM), which is shown to be strongly NP-complete by Yu et al. (2004). 

\noindent RN3DM: Given three sets, each containing $u$ integers,  $X=\{x_1, \ldots, x_u\}$, where $x_j$ is a positive integer, for $j=1, \ldots, u$, $Y=Z=\{1, 2, \ldots, u\}$,   and a positive integer $b$, where $ub = u(1+u) + \sum_{j=1}^u x_j$, do there exist two permutations of $(1, 2, \ldots, u)$, denoted as $([1], \ldots, [u])$ and $(<1>, \ldots, <u>)$, such that $x_j + y_{[j]}+ z_{<j>}=x_j + [j] + <j> = b$ for $j=1, \ldots, u$?

Given an instance of RN3DM, we construct an instance of our problem 
as follows: $n=u$; $S_{j1}-\beta=b-x_j$ if $\beta_j\equiv \beta$,  or $S - \beta_j=b-x_j$ if $S_{j1}\equiv S$; $D_1=D_2=D$ for some large positive integer $D$ such that $vD$ is a positive integer; $\lambda_j=(n+1-j)/(vD)$, for $j=1, \ldots, n$, where $v=\frac{\delta}{C}.$

For conciseness, we do not consider the two cases $S_{j1}\not\equiv S$ or $\beta_j\not\equiv \beta$ separately. Instead, we present a single proof that works for both cases. To this end, we use the generic symbol $S_{j1}$ to represent $S$ in case $S_{j1}\equiv S$, and use $\beta_j$ to represent $\beta$ in case $\beta_j\equiv \beta$. It can be seen that in this instance, $S_{j1}= b - x_j + \beta_j$ and $v\sum_{i=1}^2D_{i}\lambda_{j_i} = (n+1-j_1) + (n+1-j_2)$, for $j=1, \ldots, n.$

\noindent $(\Rightarrow)$. Given a solution to the instance of RN3DM, i.e., two permutations $([1], \ldots, [u])$ and $(<1>, \ldots, <u>)$ such that $x_j + [j]+ <j>=b$ for $j=1, \ldots, u$, we assign vehicle $j$ to position $n+1-[j]$ in segment 1 (i.e., $j_1=n+1-[j]$) and position $n+1-<j>$ in segment 2 (i.e., $j_2=n+1-<j>$). This implies that the final SoC of vehicle $j$ is: 
\vspace{-0.1in}
\begin{eqnarray*}
    S_{j1} - v\sum_{i=1}^2D_{i}\lambda_{j_i} & = & b -  x_j  +\beta_j - (vD)(\lambda_{n+1-[j]}+ \lambda_{n+1-<j>}) \\
    & = & b - x_j + \beta_j - ([j] + <j>) = \beta_j, ~~\mbox{for $j=1, \ldots, n$}
\end{eqnarray*}
Thus, all the constraints are satisfied. 

\noindent $(\Leftarrow)$. Given a feasible solution to the instance of our problem, i.e.,  vehicle $j=1, \ldots, n$ is assigned to some position $[j]$ in segment 1 and some position  $<j>$ in segment 2, such that  $S_{j1} - vD(\lambda_{[j]} + \lambda_{<j>}) \ge \beta_j$, which implies that  $b-x_j + \beta_j - ((n+1-[j]) + (n+1-<j>)) \ge \beta_j$, and hence 
\vspace{-0.1in}
\begin{equation}
x_j + (n+1- [j]) + (n+1-<j>) \le b, ~~\mbox{for $j=1, \ldots, n$}. \label{only-if-eq1}
\end{equation}
Moreover, since the total of the left-hand-sides of these inequalities is
\vspace{-0.1in}
\[
\sum_{j=1}^n (x_j + (n+1-[j]) + (n+1-<j>)) =\sum_{j=1}^n x_j + \sum_{j=1}^n j + \sum_{j=1}^n j =  nb,
\]
in order for all these inequalities in (\ref{only-if-eq1}) to hold, they must all be satisfied at equality. Thus, $x_j + (n+1-[j]) + (n+1-<j>) = b$, for $j=1, \ldots, n$. Since $(n+1-[1], \ldots, n+1-[n])$ and $(n+1-<1>, \ldots, n+1-<n>)$ are both permutations of $(1, \ldots, n)$, the $n$ equations:  $x_j+(n+1-[j])+(n+1-<j>) = b$, for $j=1, \ldots, n$, form a solution to the instance of RN3DM.  \blot

\subsubsection{Problems $n,2~|~C_j=C,S_{j1}=S_1, \beta_j=\beta~|~TU$ and $n,2~|~\delta_j=\delta,S_{j1}=S_1, \beta_j=\beta~|~TU$}

In these problems, since the $\delta_j$ or $C_j$ values can be different, the vehicles are not identical. Thus, unlike in problem $n,2~|~C_j=C,\delta_j=\delta~|~TU$, different feasible solutions for these problems may have different objective values. Hence, finding a feasible solution is not sufficient. However, doing so is already difficult, as we now demonstrate.





\begin{proposition} \label{prop2.2.3}
Finding a feasible solution to problem $n,2~|~S_{j1}=S_1, \beta_j=\beta~|~TU$ is strongly NP-hard, even when $C_j = C$ or  $\delta_j = \delta$.
\end{proposition}
{\bf Proof:} 
We show that the feasibility version of problem $n,2~|~C_j=C, S_{j1}=S_1, \beta_j=\beta~|~TU$ is strongly NP-hard by a reduction from the restricted variant of Numerical 3-D Matching problem (RN3DM), described in Section~\ref{sec2.2.2}. A similar proof can be given for the feasibility version of problem $n,2~|~\delta_j=\delta, S_{j1}=S_1, \beta_j=\beta~|~TU$, and hence is omitted. Given an instance of RN3DM, we construct an instance of the feasibility version of problem $n,2~|~C_j=C, S_{j1}=S_1, \beta_j=\beta~|~TU$ as follows:   $n=u$, three integers $C$, $S_1$ and $\beta$ with $S_1>\beta$, $C_j= C$, $S_{j1}=S_1$ and $\beta_j=\beta$, for $j=1\ldots, n$; $D_1=D_2=D$ for some large positive integer $D$ that is much larger than $C$, and  $\lambda_j=(n+1-j)C/D$,  and $\delta_j = \frac{S_1-\beta}{b-x_j}$,  for $j=1, \ldots, n$. 

It can be seen that in this instance, $D_1\delta_j\lambda_{j_1} + D_2\delta_j\lambda_{j_2} = \frac{S_1-\beta}{b-x_j}(j_1 + j_2)C$, for $j=1, \ldots, n.$ We show that there exists a solution to RN3DM if and only if there exists a solution to the constructed instance of problem $n,2~|~C_j=C, S_{j1}=S_1, \beta_j=\beta~|~TU$.

\noindent $(\Rightarrow)$ Given a solution to the instance of RN3DM, i.e., two permutations $([1], \ldots, [u])$ and $(<1>, \ldots, <u>)$ such that $x_j + [j]+ <j>=b$ for $j=1, \ldots, u$, we assign vehicle $j$ to position $n+1-[j]$ in phase 1 (i.e., $j_1=n+1-[j]$) and position $<j>$ in phase 2 (i.e., $j_2=n+1-<j>$). This implies that the energy usage by vehicle $j$ is: 
\vspace{-0.1in}
\begin{eqnarray*}
    D_1\delta_j\lambda_{j_1} + D_2\delta_j\lambda_{j_2} & = & \frac{S_1-\beta}{b-x_j}([j]+<j>)C = (S_1-\beta)C, ~~\mbox{for $j=1, \ldots, n$}
\end{eqnarray*}
Thus, the final SoC of each vehicle is $S_1 - [(S_1-\beta)C]/C = \beta$, and hence all the constraints are satisfied.

\noindent $(\Leftarrow)$ Given a solution to the instance of our problem, i.e.,  vehicle $j=1, \ldots, n$ is assigned to some position $[j]$ in segment 1 and some position  $<j>$ in segment 2, such that all the constraints are satisfied, i.e., 
\vspace{-0.1in}
\begin{eqnarray}
 D\delta_j\lambda_{[j]}/C  + D\delta_j\lambda_{<j>}/C & \leq & S_1 - \beta, ~~\mbox{for $j=1, \ldots, n$}.
\end{eqnarray}
By the definition of $\delta_j$ and $\lambda_j$, the above inequalities imply that 
\[
\frac{S_1-\beta}{b-x_j}((n+1-[j])+(n+1-<j>)) \le S_1 - \beta, ~~\mbox{for $j=1, \ldots, n$}.
\]
Thus, $x_j+(n+1-[j])+(n+1-<j>) \le b$, for $j=1, \ldots, n$. Given that the sum of the left-hand sides of all these inequalities is $nb$, these inequalities must all be satisfied at equality. Thus,  $x_j+(n+1-[j])+(n+1-<j>) = b$, for $j=1, \ldots, n$. Since $(n+1-[1], \ldots, n+1-[n])$ and $(n+1-<1>, \ldots, n+1-<n>)$ are both permutations of $(1, \ldots, n)$, the $n$ equations:  $x_j+(n+1-[j])+(n+1-<j>) = b$, for $j=1, \ldots, n$, form a solution to the instance of RN3DM. \blot

\subsection{Problem with Fixed $m\ge 3$ and Arbitrary $n$} \label{sec4-openprob}
From the results shown above, even the feasibility version of the problem with an arbitrary $n$ and a fixed $m\ge 3$  where at least one set of the parameters are non-identical, i.e., problem $n,\bar{m}\ge 3\mid\mid TU$, is  strongly NP-hard. 
However, the complexity of both the feasibility version and the optimization version of this problem where every set of parameters are identical, i.e., problem $n,\bar{m}\ge 3\mid C_j=C, \delta_j=\delta, S_{j1}=S_1, \beta_j=\beta \mid TU$,  remains open. We observe that finding a feasible solution to this problem is as hard as a special case of the numerical 3D matching (N3DM) problem  where the three sets involved are identical (i.e., $X=Y=Z$), which we denote by N3DM-IDEN. While N3DM is known to be strongly NP-hard (Garey and Johnson 1979), to the best of our knowledge, N3DM-IDEN is an open problem. We conjecture that N3DM-IDEN is strongly {\em NP}-hard, which would imply that the feasibility version of our problem $n,\bar{m}\ge 3\mid C_j=C, \delta_j=\delta, S_{j1}=S_1, \beta_j=\beta \mid TU$ is also strongly {\em NP}-hard. 


\subsection{Problem with Arbitrary $m$ and Fixed $n$} \label{sec4.3}
We first show that the problem with an arbitrary $n$ and a fixed $m$ is ordinarily {\em NP}-hard by proving this result for a special case with two identical vehicles ($n = 2$, $C_1=C_2\equiv C$, $\delta_1=\delta_2\equiv \delta$), equal initial states of charge (i.e. $S_{11}=S_{21}$), and a zero  charge level required for both vehicles at the end (i.e., $\beta_1=\beta_2=0$).

\begin{proposition}  \label{prop:2.3bnpc}
Finding a feasible solution to problem $2, m~|~C_j=C, \delta_j=\delta, S_{j1}=S_1, \beta_j=\beta~|~TU$ is ordinarily NP-hard.
\end{proposition}
{\bf Proof:} By reduction from the Partition problem, which is known to be binary {\em NP}-complete (Karp 1972) and defined as follows.

\noindent {\em Partition:} Given $u$ positive integers, $x_1, \ldots, x_{u}$ such that $\sum_{j=1}^{u} x_j = 2X$ for some integer $X$, does there exist a Partition $A_1$ and $A_2$ of the set $A = \{1,\ldots, u\}$ such that $\sum_{j\in A_1} x_j = \sum_{j\in A_2} x_j = X$? 

Given an arbitrary instance of Partition, we construct an instance of problem $\bar n, m~|~|TU$ as follows: $n = 2$, $m=u$; $C_1=C_2\equiv C$ for some positive integer $C>3X$, $\delta_1=\delta_2\equiv \delta=\frac{1}{C}$, $\eta_1=1/2$, $\eta_2=3/4$; $D_{i} = 4Cx_i$, for $i=1, \ldots, m$; and 
$S_{11} = S_{21} = \frac{3X}{C}, \beta_1 = \beta_2=0$. We show that this instance of the problem has a feasible solution if and only if there exists a Partition.

\noindent ($\Rightarrow$) If there exists a Partition $A_1, A_2$, then in phase $i$ for $i\in A_1$, we assign vehicle 1 to position 1 and vehicle 2 to position 2, and in phase $i$ for $i\in A_2$, we assign vehicle 1 to position 2 and vehicle 2 to position 1. Thus, the total energy usages of the two vehicles are, 
\[
\sum_{i\in A_1}^m D_i\delta(1-\eta_1) + \sum_{i\in A_2}^m D_i\delta(1-\eta_2) = \sum_{i\in A_1} 4Cx_i(\delta/2) + \sum_{i\in A_2} 4Cx_i(\delta/4)
= 2 \sum_{i\in A_1} x_i + \sum_{i\in A_2} x_i = 3X, \quad \mbox{and}
\]
\[
\sum_{i\in A_1}^m D_i\delta(1-\eta_2) + \sum_{i\in A_2}^m D_i\delta(1-\eta_1) = \sum_{i\in A_1} 4Cx_i(\delta/4) + \sum_{i\in A_2} 4Cx_i(\delta/2)
= \sum_{i\in A_1} x_i + 2\sum_{i\in A_2} x_i = 3X.
\]
Given that their initial states-of-charge, $S_{11}=S_{21}=3X/C$, after the usage of $3X$, each vehicle's ending state of charge becomes 0. Therefore, a feasible solution is found for both vehicles and the total energy usage is $6X$, as required.

\noindent ($\Leftarrow$) Assume that there exists a solution, denoted by $\sigma$, that is feasible for both vehicles. 
Observe that, in any given solution including $\sigma$, the total energy usage by the two vehicles over any given road segment, say, segment $i$ (from location $i$ to $i+1$ with distance $D_i$) is $D_i\delta(1-\eta_1) + D_i\delta(1-\eta_2) = (4Cx_i/C)(1/4+1/2) = 3x_i$, which is independent of how the two vehicles are sequenced. This implies that the total energy usage by the two vehicles in any given solution is $\sum_{i=1}^m 3x_i = 6X$. Since each vehicle has an initial level of energy, $(S_{11}=S_{21})C = 3X$,  in any feasible solution,  the total energy usage by each vehicle must be exactly $3X$. 

In $\sigma$, let $B_1$ denote the set of phases where vehicle 1 is assigned to the first position, and $B_2$ the set of phases where vehicle 2 is assigned to the first position. Then, the total energy usages by the two vehicles are
\[
\sum_{i\in B_1}^m D_i\delta(1-\eta_1) + \sum_{i\in B_2}^m D_i\delta(1-\eta_2) = 3X, \quad \mbox{and}
\]
\[
\sum_{i\in B_1}^m D_i\delta(1-\eta_2) + \sum_{i\in B_2}^m D_i\delta(1-\eta_1) = 3X. 
\]
These relations imply that 
\[
 2 \sum_{i\in B_1} x_i + \sum_{i\in B_2} x_i = 3X, \quad \mbox{and}~~~~\sum_{i\in B_1} x_i + 2\sum_{i\in B_2} x_i = 3X.
\]
Therefore, $\sum_{i\in B_1} x_i = \sum_{i\in B_2} x_i = X$, and hence $B_1, B_2$ is a Partition. \blot 

\bigskip

We now consider a more general version of this problem with fixed $n$ and general $m$ and general problem parameters, i.e., problem $\bar{n},m\mid \mid TU$. 
Let $\Omega \in {\cal R}^n$ denote the set of all possible positions of the $n$ vehicles in a given phase, where $\{\omega_1,\ldots,\omega_n\} \in \Omega$ specifies the positions of the vehicles. 

We describe a pseudo-polynomial time dynamic programming algorithm to solve this problem.

\medskip

\noindent {\bf Algorithm~DP1}

\noindent {\em Value Function} \\
$f_t(L_{1},\ldots, L_{n})$ = the minimum total energy usage over segments $1,\ldots,t$, given that the total energy used by vehicle $k$ over the first $t$ segments is $L_{k}$, for $k = 1,\ldots,n, \; t = 1,\ldots,m$.

\noindent {\em Boundary Condition} \\
$f_0(0, \ldots, 0) = 0$. 

\noindent {\em Optimal Solution Value} \\
$\min 
\{f_m(L_{1},\ldots,L_{n})\mid L_{k} \le (S_{k1}-\beta_k)C_k, \; \mbox{for $1 \le k \le n$} \}$. 

\medskip

\noindent {\em Recurrence Relation} \\
For $t=1,\ldots, m$, and $L_k=0,  \ldots, (S_{k1}-\beta_k)C_k$, for $k=1, \ldots, n$: 
\vspace{-0.1in}
\[f_t(L_{1},\ldots,L_{n}) 
= \min_{(\omega_1,\ldots,\omega_n)\in \Omega}
\left\{ D_t \sum_{k=1}^n (1-\eta_{\omega_k})\delta_k  + f_{t-1}(L_{1} - (1-\eta_{\omega_1})\delta_1D_t, \ldots, L_{n} - (1-\eta_{\omega_n})\delta_nD_t ) \right\},\]
where $\Omega$ is the set of all permutations of $(1, \ldots, n)$.

\medskip

Algorithm~DP1 starts with no vehicles scheduled. At each phase, the recurrence relation considers all possible positions for the $n$ vehicles. 

\begin{proposition} \label{prop:FJ}
Algorithm~DP1 solves  problem $\bar n,m|~|~TU$ in $O(mn n!Q_{\max}^n)$ time by either showing the infeasibility of the problem or finding an optimal solution if the problem is feasible.
\end{proposition}
{\bf Proof:} Algorithm~DP1 compares the total energy usage of all possible vehicle positions at all phases, and therefore finds an optimal schedule. In the algorithm, the state space enumerated is $O(m Q_{\max}^n)$. The algorithm considers all possible $n!$ sequences for the $n$ vehicles at each of the $m$ phases. For each possible sequence, the remaining energy level of each vehicle is updated. Thus, the computation time at each phase is bounded by $O(n n!)$. Therefore, the overall computation time of the algorithm is bounded by $O(mn n!Q_{\max}^n)$. Since $n$ is fixed, this computation time is pseudo-polynomial. \blot

\medskip

We now describe a fully polynomial time approximation scheme (FPTAS), i.e. a family of algorithms DP1$_{\epsilon}$ parametrized by $\epsilon$, for problem~$\bar n,m|~|~TU$. The state variables are $\bigr (t, L_1, \ldots, L_n \bigr )$. The trial state generation step appends a sequence of vehicle positions for segment $t + 1$ by creating $n!$ trial states corresponding to all possible vehicle position sequences at segment $t + 1$. The trial state labeling step attaches labels $\bigl (t+1, \Gamma(L_1), \ldots, \Gamma(L_n) \bigr )$ that are close in value to the corresponding state variables $\bigl (L_1, \ldots, L_n\bigr )$. Specifically, the function $\Gamma(\cdot)$ is defined by $\Gamma(x) = \gamma^h$ when a value of $x$ satisfies $\gamma^h \le x < \gamma^{h+1}$, for some non-negative integer $h$, and $\gamma = 1 + \epsilon/(2m)$. When multiple states have the same label, all of them except one are eliminated, thereby ensuring that at most one state is retained corresponding to each label. Finally, the total energy usage is computed as the sum of the individual vehicle usages, and the minimum total and its corresponding solution are found. Let $\Omega$ denote the set of all $n!$ possible position sequences of the $n$ vehicles, and $\omega_{ji}$ denote the position of vehicle $j$ in sequence $i \in \Omega$.

\medskip

\noindent {\bf Algorithms~DP1{$_{\epsilon}$}}

\noindent {\em State Variables.} Let $\bigl [t,L_1,\ldots, L_n \bigr ]$ correspond to a partial solution for segments $1,\ldots,t$, where vehicle $j$ has used energy $L_j \ge 0$, for $j = 1,\ldots, n$ over those segments.

\noindent {\em Initialization.} \\
Define the initial state $[0, 0, \ldots, 0]$. \\
Set $j = 0$, and $\gamma = 1 + \epsilon /(2m)$.

\noindent {\em Trial State Generation.} For each state
$\bigl [t,L_1,\ldots,L_n \bigr ]$, generate the $n!$ trial states \\
$\bigl [t+1, L_1 + D_{t+1} \delta_1 (1 - \eta_{\omega_{1},1}),
\ldots, L_n + D_{t+1} \delta_n (1 - \eta_{\omega_{n},1} \bigr]$, 
through \\
$\bigl [t+1, L_1 + D_{t+1} \delta_1 (1 - \eta_{\omega_{1},n!}),
\ldots, L_n + D_{t+1} \delta_n (1 - \eta_{\omega_{n},n!}\bigr ]$. \\
However, the state 
$\bigl [t+1, L_1 + D_{t+1} \delta_1 (1 - \eta_{\omega_{1},i}),
\ldots, L_n + D_{t+1} \delta_n (1 - \eta_{\omega_{n},i})\bigr ]$
corresponding to sequence $i \in \Omega$ is only generated when $L_j + D_{t+1} \delta_j (1 - \eta_{\omega_{ji}}) \le ( S_{j1} - \beta_j)C_j$, for $j = 1,\ldots,n$.

\noindent {\em Trial State Labeling.} For each trial state
$\bigl [t+1, L_1, \ldots, L_n \bigr ]$, attach the label
$\bigl (t+1, \Gamma(L_1), \ldots, \Gamma(L_n) \bigr )$, where $\Gamma(x) = \gamma^h$ when $\gamma^h \le x < \gamma^{h+1}$, for some non-negative integer $h$.

\noindent {\em Trial State Elimination.} For all trial states with identical labels, eliminate all except one with lexicographically smallest value of $L_1, \ldots, L_n$, choosing arbitrarily in the event of a tie.

\noindent {\em Termination Test.} If $t + 1 < m$, then set $t  = t + 1$ and return to the trial state generation step. Otherwise, select a state $\bigl [m, \tilde{L}_1, \ldots, \tilde{L}_n \bigr ]$ where $\tilde{L}_1 + \ldots + \tilde{L}_n$ is smallest, and backtrack to find the corresponding schedule.

\medskip

We now establish the accuracy and running time of Algorithms~DP1{$_{\epsilon}$}. Let $D_{\max} = \max_{1 \le t \le m} \{D_t\}$ and $\delta_{\max} = \max_{1 \le j \le n} \{\delta_j\}$.

\begin{proposition} \label{DP1fptas}
For any given $\epsilon > 0$, if Algorithms~DP1{$_{\epsilon}$} generates a feasible solution for problem~$\bar n,m|~|~TU$, then it is a $(1+\epsilon)$-approximation solution for the problem. The algorithm has an $O(m \cdot [(1 + 2m/\epsilon) \cdot
\ln{(m \cdot D_{\max} \cdot \delta_{\max}})]^n \cdot n!)$ running time.
\end{proposition}
{\bf Proof:} We observe that the condition in the trial state generation step ensures that only feasible solutions are generated. First, we analyze the cost of the solution delivered by Algorithm~DP1{$_{\epsilon}$}. The proof works by induction on the number of segments that have been considered for the vehicle positioning decision. Our induction hypothesis is that, given any state $\bigl [t,L_1,\ldots, L_n \bigr ]$ generated by  Algorithm DP1, 
Algorithm~DP1{$_{\epsilon}$} generates a state 
$\bigl [t,\tilde L_1,\ldots, \tilde L_n \bigr ]$
where $\tilde L_j \le \gamma^t L_j$, for $j = 1,\ldots, n$. The hypothesis holds for $j = 0$. Now, suppose that the hypothesis holds for $j = 0, 1, \ldots, k$. Let $[k, L_1,\ldots, L_n \bigr ]$ denote a state in the exact dynamic program.

We assume that the exact dynamic program sequences the vehicles following sequence $i \in \Omega$ at segment $k + 1$, resulting in a state
$\bigl [k + 1, L'_1,\ldots, L'_n \bigr ]$. Now, from the induction hypothesis, we have a state 
$\bigl [k, \tilde L_1,\ldots, \tilde L_n \bigr ]$ in DP1{$_{\epsilon}$}. 
where $\tilde L_j \le \gamma^k L_j$, for $j = 1,\ldots, n$.
The trial state, $S$, that is generated from 
$\bigl [k, \tilde L_1,\ldots, \tilde L_n \bigr ]$ by using sequence $i \in \Omega$ in segment $k + 1$ is denoted
by $[k + 1, \tilde L'_1, \ldots, \tilde L'_n]$.

Observe that the total energy usage of each vehicle $j$
is increased by $D_{k+1}\delta_j(1 - \eta_{\sigma(j)})$ in both the exact and approximate algorithms. Therefore, for $j = 1, \ldots, n$, and using the induction hypothesis, we have
\vspace{-0.2in}
\begin{eqnarray*}
\tilde L_j & \le & \gamma^k L_j \\ 
\Rightarrow \tilde L_j + D_{k+1}\delta_j(1 - \eta_{\sigma(j)}) & \le & \gamma^k L_j
+ D_{k+1}\delta_j(1 - \eta_{\sigma(j)}) \\
\Rightarrow \tilde L'_j & \le & \gamma^k 
L'_j. 
\end{eqnarray*}
Thus, in the case where state $S$ exists in Algorithm~DP1{$_{\epsilon}$}, the induction argument is complete.

However, state $S$ may be eliminated by another trial state $S' =[k + 1, \tau_1, \ldots, \tau_n]$. In this case, since the states $S$ and $S'$ share the same label, we have $\tau_j \le \gamma \tilde L'_j \le \gamma \cdot \gamma^k L'_j = \gamma^{k+1} L'_j$, for $j = 1,\ldots,n$. Thus, the induction hypothesis is proved.

We now show that Algorithm~DP1{$_{\epsilon}$} delivers a solution with the required performance guarantee. Let the state $[m, L^*_1, \ldots, L^*_n]$ correspond to an optimal solution with cost $\sum_{j=1}^n L^*_j$. Then, it follows from the above induction argument that 
Algorithm~DP1{$_{\epsilon}$} generates a solution with cost $\sum_{j=1}^n \tilde L_j \le \gamma^m \sum_{j=1}^n L^*_j$. Therefore,
\vspace{-0.1in}
\begin{eqnarray*} 
(\sum_{j=1}^n \tilde L_j - \sum_{j=1}^n L^*_j) / \sum_{j=1}^n L^*_j & \le &  \gamma^m - 1 \le  \epsilon, 
\end{eqnarray*}
where the last inequality follows from (i) the specification of $\Gamma(\cdot)$ in the Initialization step, and (ii) the fact that $\gamma^m = [1 + \epsilon/(2m)]^m$ is a convex function of $\epsilon$ over the range $0 \le \epsilon \le 2$ and the inequality $\gamma^m \le 1 + \epsilon$ holds at both ends of that range. 

Finally, we analyze the time complexity of Algorithm~DP1{$_{\epsilon}$}. The labels are of the form $(j + 1, \Gamma(L_1), \ldots, \Gamma(L_n))$, where $j \le m-1$, and 
$L_j \le m \cdot D_{\max} \cdot \delta_{\max}$, for $j = 1,\ldots,n$, even considering no energy saving in the platoon. Hence, the number of possible labels $\Gamma(L_j)$, for $j = 1,\ldots,n$, is given by
\vspace{-0.2in}
$$
\lceil \log_{\gamma} (m \cdot D_{\max} \cdot \delta_{\max}) \rceil
= \lceil \ln{( m \cdot D_{\max}\cdot \delta_{\max})} / \ln{\gamma} \rceil \le \lceil (1 + 2m/\epsilon) \cdot
\ln{(m \cdot D_{\max} \cdot \delta_{\max})}\rceil,
$$
where the last inequality is obtained from (i) the specification of $\gamma = 1 + \epsilon/2m$, and (ii) the inequality $\ln(x) \ge (x-1)/x$, for $x \ge 1$, which implies $1 / \ln{\gamma} \le \frac{1 + \epsilon/(2m)}{\epsilon/(2m)} = 1 + 2m/\epsilon$.

Hence, the overall number of labels is $O(m \cdot [(1 + 2m/\epsilon) \cdot
\ln{(m \cdot D_{\max} \cdot \delta_{\max})}]^n)$. Each state generates at most $O(n!)$ trial states. Then, it follows that the overall time complexity of Algorithm~DP1{$_{\epsilon}$} is 
$O(m \cdot [(1 + 2m/\epsilon) \cdot
\ln{(m \cdot D_{\max} \cdot \delta_{\max})}]^n \cdot n!)$. \blot

\subsection{Problem with Arbitrary $n$ and $m$} \label{sec2.5}

We show that the feasibility version of the problem with an arbitrary $m$ and an arbitrary $n$ is strongly NP-hard even when the vehicles are identical (i.e., $C_j\equiv C$ and $\delta_j\equiv \delta$) with identical initial SoCs (i.e., $S_{j1} = S_1$) and identical lowest allowable SoCs (i.e., $\beta_j = \beta$).

\begin{proposition}  \label{prop:2.5-snpc}
Finding a feasible solution to problem $n,m~|~C_j=C, \delta_j=\delta, S_{1j}=S_1,\beta_j=\beta~|~TU$ is strongly NP-hard.
\end{proposition}
{\bf Proof:} By reduction from the 3-Partition problem, which is known to be unary {\em NP}-complete (Garey and Johnson 1979) and defined as follows.

\noindent {\em 3-Partition:} Given $3u$ positive integers, $x_1, \ldots, x_{3u}$ such that $\frac{X}{4}<x_i < \frac{X}{2}$ for $i=1, \ldots, 3u$ and $\sum_{i=1}^{3u} x_i = uX$ for some integer $X$, does there exist a partition of $A=\{1, \ldots, 3u\}$ into $u$ distinct subsets $A_1, \ldots, A_u$ such that each subset contains exactly 3 elements which sum  to exactly $X$, i.e., $|A_j|=3$ and $\sum_{k\in A_j} x_k = X$, for $j=1, \ldots, u$? 

Given an arbitrary instance of 3-Partition, we construct an instance of the feasibility version of problem 1, as follows: $n = u$, $m=3u$; $C_j\equiv C$ for some large positive integer $C>2uX$, $\delta_j\equiv \delta=\frac{1}{C}$, $\eta_1=\cdots=\eta_{u-1}=0$; $\eta_u =\frac{X-1}{X}$, $D_{i} = Cx_i$, for $i=1, \ldots, m$; and $\beta_{j} = \beta = \frac{(u-1)X+1}{C}$ and $S_{j1} = 2\beta$, 
 for $j=1, \ldots, n$.  We show that this instance of our problem has a feasible solution if and only if there exists a solution to the 3-Partition instance. 

\noindent ($\Rightarrow$) If there exists a partition $A_1, \ldots, A_u$, then for $j=1, \ldots, u$, assign vehicle $j$ to position $u$ in phases  $i\in A_j$, and any position $k<u$ in all other phases. 
Thus, the total energy usage of each vehicle $j$ in all phases, denoted as $Q_j$, is
\begin{eqnarray*}
Q_j & = & \sum_{i\in A_j} D_i\delta(1-\eta_u) + \sum_{i\in A\setminus A_j} D_i\delta \\
& = & \sum_{i\in A_j} (Cx_i)\frac{1}{C}\frac{1}{X} + \sum_{i\in A\setminus A_j} Cx_i\frac{1}{C} = 1 + (u-1)X = \beta C,
\quad \mbox{for $j=1, \ldots, n$}.
\end{eqnarray*}
Given that their initial state of charge, $S_{j1}=2\beta$, after the usage of $Q_j$, each vehicle's ending SoC becomes $2 \beta - Q_j/C = \beta$. Therefore, a feasible solution is found.

\noindent ($\Leftarrow$) Assume that there exists a solution, denoted by $\sigma$, that is feasible with the final SoC of each vehicle at least $\beta$. 
In $\sigma$, let $B_j$ denote the set of phases where vehicle $j$ is assigned to the last position, for $j=1,\ldots, n$. Thus, the total energy usage by vehicle $j$ is 
\vspace{-0.1in}
\[
\sum_{i\in B_j} D_i\delta(1-\eta_n) + \sum_{i\in A\setminus B_j} D_i\delta = \sum_{i\in B_j} \frac{x_i}{X}
+ \sum_{i\in A\setminus B_j} x_i = uX - (1-\frac{1}{X})\sum_{i\in B_j} x_i.
\]
Since $\sigma$ is feasible, the energy usage by each vehicle $j$ must not exceed $C(S_{j1} - \beta_j) = C\beta = (u-1)X+1$. Thus, we have
\vspace{-0.1in}
\[
uX - (1-\frac{1}{X})\sum_{i\in B_j} x_i \le (u-1)X+1, \quad \mbox{for $j=1,\ldots, n$} \]
which implies that
\vspace{-0.1in}
\[
\sum_{i\in B_j} x_i \ge X, \quad \mbox{for $j=1,\ldots, n$} 
\]
This, along with the fact that 
$\sum_{i=1}^u\sum_{i\in B_j} x_i = uX$ and $\frac{X}{4}<x_i < \frac{X}{2}$ for $i=1, \ldots, 3u$, implies that $\sum_{i\in B_j} x_i = X$ and $|B_j|=3$, for $i=1, \ldots, u$. Thus, $B_1, \ldots, B_n$ is a feasible partition to 3-Partition.   \blot 

\bigskip

Proposition~\ref{prop:2.5-snpc} naturally raises the question of whether the problem becomes more tractable when there is no feasibility issue, i.e., when the problem is always feasible. To address this question, we consider the case where the capacity and the initial SoC of each vehicle are sufficiently large that any solution is feasible regardless of how the vehicles are sequenced in any segment, i.e., $S_{j1}C_j\ge \bar L$, where $\bar L$ is a sufficiently large initial energy level, e.g., $\bar L \ge \beta_jC_j + \sum_{i=1}^m D_i(1-\eta_1)\delta_j$, for  $j=1,\ldots,n$. In this case, even if a vehicle is sequenced in the first position in every segment, its final SoC is at least the minimum required. Here, the values of $\delta_j$ are the only vehicle parameters that influence the solution because they are the only parameters associated with the vehicles that are involved in the objective value of a given solution. We show that a simple rule solves the problem.

\bigskip

\noindent {\bf Rule 4:} Reindex the vehicles in non-decreasing order of $\delta_j$, and assign vehicle $j$ to the
$j$th position in every road segment, for $j = 1,\ldots, n$.

\begin{proposition} \label{prop-rule4}
    Rule 4 solves problem $n,m\mid S_{j1}C_j\ge \bar L\mid TU$ in $O(n\log n)$ time. 
\end{proposition}
{\bf Proof:} In each segment $i$, the energy usage by vehicle $j$ is $D_i\delta_j(1-\eta_{j_i})$, where $j_i$ is the position of the vehicle in segment $i$. Thus, the total usage by all the vehicles in segment $i$ is $D_i\sum_{j=1}^n \delta_j(1-\eta_{j_i})$, where $(j_1, \ldots, j_n)$ is a permutation of $(1,\ldots, n)$. By Lemma~\ref{lemma1}, the summation $\sum_{j=1}^n \delta_j(1-\eta_{j_i})$ is minimized by the assignment generated from Rule 4. \blot

\section{Minimizing Maximum Energy Usage} \label{sec:max-min}

In this section, we consider various cases of problem $n,m~|~|MU$, which minimizes the maximum energy usage of the vehicles, subject to the constraint that the final SoC of each vehicle is no less than a given minimum required level. 

\subsection{Problem with Arbitrary $n$ and $m=1$} \label{sec3.1}
Problem $n,1~|~|MU$ is to find platoon positions for the $n$ vehicles, i.e., a permutation, denoted as $([1], \ldots, [n])$, for the first road segment, with the goal of minimizing the maximum energy usage by the vehicles, i.e., $\max\{U_{2j}|j=1,\ldots, n\}$  subject to the constraint that $S_{j2}\ge \beta_j$ for $j=1,\ldots, n$, where  $U_{2j}$ and $S_{2j}$ are vehicle $j$'s energy usage and final SoC, which are defined by (\ref{Usage-def}) and (\ref{SoC-def}), respectively.

For ease of presentation, we define  $\lambda_{[j]}=1 - \eta_{[j]}$. Thus, $\lambda_1\ge \cdots \ge \lambda_n$. We can rewrite $U_{j2}$ and $S_{j2}$ as follows:
\begin{equation}
U_{j2} = D_1\delta_j\lambda_{[j]}, ~~~
S_{j2} = S_{j1} - \frac{D_1\delta_j\lambda_{[j]}}{C_j},  ~~\mbox{for $j=1, \ldots, n$}. \label{sec4.1-eq2}
\end{equation}

This problem can be solved by Algorithm Delta given in Section~\ref{sec2.1}.

\begin{proposition}  \label{prop:3.1-MaxMatch}
Algorithm Delta solves problem $n,1~|~|~MU$ in $O(n^2)$ time.
\end{proposition}
{\bf Proof:} The same arguments in the proof for Proposition~\ref{prop2.1}  can be applied directly after replacing (\ref{eq-delta}) by the following equation, which defines the reduction in the maximum energy consumption when $\pi$ is modified as stated in the proof there. 
\vspace{-0.1in}
\begin{eqnarray*}
& & D_1\max\left\{\delta_{j_{p+1}(\pi)}\lambda_{p+1}, ~\delta_{j^*_{p+1}}\lambda_{k}\right\} -  D_1\max\left\{\delta_{j_{p+1}(\pi)}\lambda_{k}, ~\delta_{j^*_{p+1}}\lambda_{p+1}\right\} \nonumber \\
& & = D_1\delta_{j_{p+1}(\pi)}\lambda_{p+1} - D_1\max\left\{\delta_{j_{p+1}(\pi)}\lambda_{k}, ~\delta_{j^*_{p+1}}\lambda_{p+1}\right\} \geq 0. 
\end{eqnarray*}
The inequality above is due to the fact that $\lambda_{p+1}\ge \lambda_k$ and $\delta_{j_{p+1}(\pi)}\ge \delta_{j^*_{p+1}}$. \blot

\bigskip

\subsection{Problem with Arbitrary $n$ and $m=2$}  \label{sec3.2}

Problem $n,2~|~|~MU$ is to find for the $n$ vehicles, positions at location 1, i.e., a permutation, denoted as $([1], \ldots, [n])$, and positions at location 2, i.e., another permutation $(<1>, \ldots, <n>)$,  to minimize the maximum usage $\max\{U_{j3} ~|~j=1\ldots, n\}$, subject to the constraint that $S_{j3}\ge \beta_j$, for $j=1, \ldots, n$, where 
\vspace{-0.2in}
\begin{eqnarray}
U_{j3} & = & D_1\delta_j(1-\eta_{[j]}) + D_2\delta_j(1-\eta_{<j>}),  ~~\mbox{for $j=1, \ldots, n$ \quad and}, \\
S_{j3} & = & S_{j1} - \frac{D_1\delta_j(1-\eta_{[j]})}{C_j} - \frac{D_2\delta_j(1-\eta_{<j>})}{C_j},  ~~\mbox{for $j=1, \ldots, n$}. \label{sec4.2-eq1}
\end{eqnarray}

We consider the same three cases of the problem parameters as in Section~\ref{sec2.2} and derive similar results. Specifically, we show that problem $n,2~|~C_j=C,\delta_j=\delta,S_{j1}=S,\beta_j=\beta~|MU$ is solvable by Rule 1, and finding a feasible solution to problems $n,2~|~C_j=C,\delta_j=\delta~|~MU$ and $n,2~|~S_{j1}=S_1, \beta_j=\beta~|~MU$ are both strongly NP-hard.

We first consider problem $n,2~|~C_j=C,\delta_j=\delta,S_{j1}=S_1,\beta_j=\beta~|MU$. We observe that different solutions for this problem may have different objective values, unlike problem $n,2~|~C_j=C,\delta_j=\delta,S_{j1}=S,\beta_j=\beta~|TU$, for which any feasible solution has the same objective value. However, as we show below, Rule 3, which solves the latter problem, also solves the current problem. 



\begin{proposition}  \label{prop:3.2-Rule1}
Rule 3 solves problem $n,2~|~C_j=C,\delta_j=\delta, S_{j1}=S, \beta_j=\beta~|~MU$ in $O(n)$ time either by finding an optimal solution or verifying that the problem is infeasible. 
\end{proposition}
\noindent {\bf Proof:}  It follows from Proposition~\ref{prop3} that if this problem is feasible, then Rule 1 generates a  feasible solution.  We prove that in this case, the solution generated by Rule 3 is optimal. 

Define $u=D_1\delta$ and $v=D_2\delta$.  Suppose vehicle $j$ is assigned to position $j_1$ in segment 1 and position $j_2$ in segment 2. Then, the total energy usage by vehicle $j$ is $u(1-\eta_{j_1}) + v(1-\eta_{j_2})$. Applying the result on the maximum  pairwise sum in Lemma~\ref{lemma1}, part (ii), it is evident that Rule 3 provides a solution with the least possible maximum energy usage of the vehicles. \blot

\medskip

We now consider problems $n,2~|~C_j=C,\delta_j=\delta~|~MU$ and $n,2~|~S_{j1}=S,\beta_j=\beta~|~MU$.  Propositions~\ref{prop2.2.2} and \ref{prop2.2.3} imply the following results immediately:

\begin{corollary}\label{prop:3.2-SNP1}
Finding a feasible solution to problem
$n,2~|~C_j=C,\delta_j=\delta~|~MU$ is strongly NP-hard, even when $S_{j1}$'s are identical or when $\beta_j$'s are identical. 
\end{corollary}  

\begin{corollary}  \label{prop:3.2-SNP2}
Finding a feasible solution to problem
$n,2~|~S_{j1}=S,\beta_j=\beta~|~MU$  is strongly NP-hard, even when $C_{j}$'s are identical or when $\delta_j$'s are identical. 
\end{corollary}  

Corollaries~\ref{prop:3.2-SNP1} and \ref{prop:3.2-SNP2} motivate us to investigate whether these problems become tractable when there is no feasibility issue, i.e., when  $S_{j1}C_j\ge \bar L$, where $\bar L$ is a sufficiently large initial energy level, e.g., $\bar L \ge \beta_jC_j + \sum_{i=1}^m D_i(1-\eta_1)\delta_j$, for each $j=1,\ldots,n$. In this case, even if a vehicle is sequenced in the first position in every segment, its final SoC is above the required threshold. 

In the following, we first show that when $\delta_j = \delta$,  the problem where there is no feasibility issue is easy to solve. 

\begin{proposition}  \label{prop:3.2-easy}
 Rule 3 generates an optimal solution for problem $n,2~|~\delta_j=\delta, S_{j1}C_j\ge \bar L~|~MU$ in $O(n)$ time.
\end{proposition}
\noindent {\bf Proof:} Clearly, Rule~3 generates a feasible solution for this problem. The proposition can be proved using the same argument as in the proof of Proposition~\ref{prop:3.2-Rule1}. \blot

\medskip

Next, we show that when $\delta_j \not = \delta$, then the problem becomes difficult even when all the other parameters for the vehicles are identical. 
\begin{proposition}  \label{prop:3.2-SNP3}
Problem $n,2~|C_j=C, S_{j1}=S, \beta_j=\beta, SC\ge \bar L~|~MU$   is strongly NP-hard. 
\end{proposition}
{\bf Proof:} We further assume that $D_1=D_2=D$ for some integer $D$. 
Define  $P_{k}=D(1 - \eta_{k})$ for position $k=1,\ldots, n$. 

We show the proposition 
by a reduction from RN3DM, which is described in Section~\ref{sec2.2.2}.
Given an instance of RN3DM, we construct an instance of our problem as follows: $n=u$, $m = 2$; $C_j=C$, $S_{j1}=S_1$ and $\beta_j=\beta$ such that $SC\ge \bar L$, $\eta_j=1-\frac{n+1-j}{D}$, which implies that $P_j=n+1-j$, and $\delta_j= \frac{\Gamma}{b-x_j}$, for $j=1, \ldots, n$, where $\Gamma$ is the threshold of the maximum energy usage of the vehicles. We show that RN3DM has a solution if and only if the constructed instance of Problem $n,2~|C_j=C, S_{j1}=S, \beta_j=\beta, SC\ge \bar L~|~MU$ has a solution.

\noindent $(\Rightarrow)$ Given a solution to the instance of RN3DM, i.e., two permutations $([1], \ldots, [u])$ and $(<1>, \ldots, <u>)$ such that $x_j + [j]+ <j>=b$ for $j=1, \ldots, u$, we assign vehicle $j$ to position $n+1-[j]$ in segment 1 and position $n+1-<j>$ in segment 2. This implies that the total energy usage by vehicle $j$ is
\vspace{-0.1in}
\begin{eqnarray*}
U_{j3} & =  \delta_j(P_{n+1-[j]} + P_{n+1-<j>}) = \frac{\Gamma}{b-x_j}([j] + <j>)  = \Gamma, ~~\mbox{for $j=1, \ldots, n$}.
\end{eqnarray*}

Thus, the maximum usage, $\max\{U_{13}, \ldots, U_{n3}\}=\Gamma$, and hence this solution for the instance of our problem has an objective value no greater than the threshold. 

\noindent $(\Leftarrow)$ Given a solution to the instance of our problem with an objective value at most $\Gamma$, where vehicle $j$ is assigned to position $j_1$ and $j_2$ in segment 1 and segment 2, respectively,  we have: $U_{j3} = \delta_j(P_{j_1} + P_{j_2}) = \frac{\Gamma}{b-x_j}((n+1-j_1) + (n+1-j_2)) \le \Gamma$, for $j=1, \ldots, n$. This implies that 
\vspace{-0.15in}
\begin{equation}
(n+1-j_1) + (n+1-j_2) \le b-x_j, ~~\mbox{for $j=1, \ldots, n$}. \label{sec4.2-P-ineq}
\end{equation}
This, together with the fact that $n=u$ and 
\vspace{-0.1in}
\[\sum_{j=1}^u (x_j+(n+1-j_1) +(n+1- j_2)) = \sum_{j=1}^u x_j + \sum_{j=1}^u j_1 + \sum_{j=1}^u j_2 = ub,\] 
implies that every inequality in (\ref{sec4.2-P-ineq}) is satisfied as an equality. Thus, 
$x_j+(n+1-j_1) + (n+1-j_2) = b$, for $j=1, \ldots, n$. 
This implies that the two permutations $(n+1-1_1, \ldots, n+1-n_1)$ and $(n+1-1_2, \ldots, n+1-n_2)$ form a solution to the RN3DM instance. \blot

\subsection{Problem with Fixed $m\ge 3$ and Arbitrary $n$} \label{sec5-openprob}

From the results shown above, even the feasibility version of problem $n,\bar{m}\ge 3\mid \mid MU$ is strongly NP-hard. However, the complexity of both the feasibility and the optimization version of this problem  when every set of parameters are identical, i.e., problem $n,\bar{m}\ge 3\mid C_j=C, \delta_j=\delta, S_{j1}=S_1, \beta_j=\beta \mid MU$, remains open. Since the feasibility version of problem $n,\bar{m}\ge 3\mid C_j=C, \delta_j=\delta, S_{j1}=S_1, \beta_j=\beta \mid MU$ is the same as the feasibility version of problem $n,\bar{m}\ge 3\mid C_j=C, \delta_j=\delta, S_{j1}=S_1, \beta_j=\beta \mid TU$, the observations in Section~\ref{sec4-openprob} about problem N3DM-IDEN also apply here.

\subsection{Problem with Fixed $n$ and Arbitrary $m$} \label{sec5.3}

The proof of Proposition~\ref{prop:2.3bnpc} in Section~\ref{sec4.3} can be used to show that the problem with $n=2$ and an arbitrary $m$,  where the vehicles are identical with identical initial SoCs, is ordinarily NP-hard. In that proof, we can use $3X$ as the threshold for the maximum energy usage of the vehicles. 

\begin{proposition}  \label{prop:3.4-ONP}
For problem $2, m~|~C_j=C, \delta_j=\delta, S_{j1}=S_1, \beta_j=\beta|~MU$, both finding a feasible solution, and  finding an optimal solution if it is feasible, is ordinarily NP-hard.
\end{proposition}

Next, we give a pseudo-polynomial time algorithm to solve the general case of the problem $\bar n, m~|~|~MU$.
The energy level of any vehicle at any location is no more than $Q_{\max}= \max\{S_{j1}C_j ~|~ j=1, \ldots, n\}$. We define a given vector of possible energy levels of the vehicles at a given location as a feasible state. We provide an algorithm to generate all feasible states of the SoC values of the vehicles at every location. 

\medskip

\noindent {\bf Algorithm EnergyState} \\
\noindent {\bf Step 0:} Initially, at location 1, there is only one state of the energy levels of the vehicles, which is $(L_{11}, \ldots, L_{n1})$, where $L_{j1} = S_{j1}C_j$, for $j=1, \ldots, n$.

\noindent {\bf Step 1:} Enumerate all $n!$ possible ways of sequencing the $n$ vehicles at location 1. For each possible sequence $([1], \ldots, [n])$, calculate the resulting energy levels
of the vehicles when they arrive at location 2, which are $L_{j2} = L_{j1} - D_1\delta_j(1-\eta_{[j]})$. Thus, there are at most $\min\{n!, Q_{\max}^n\}$ feasible states of energy levels of the vehicles at location 2.
 For each state, record the corresponding sequence of the vehicles at location 2. 
 
 \noindent {\bf Step 2:} Now consider every location $i$ with $3\leq i\leq m+1$. Given any feasible state of energy levels of the vehicles at location $i-1$, enumerate all $n!$ ways of sequencing the $n$ vehicles at location $i-1$. Under each possible sequence, calculate the resulting energy levels of the vehicles
  at location $i$. Since there are at most  $Q_{\max}^n$ feasible states of energy levels of the vehicles at location $i$, when going from each state at location $i-1$ to location $i$ under $n!$ possible sequences, it is possible that multiple states at  location $i-1$ may lead to the same state at location $i$.
For each state at location $i$, we  record the corresponding state of energy levels and the sequence of the vehicles at location $i-1$ that lead to this state at location $i$.

\medskip

\begin{proposition}  \label{prop:3.4-Enumerate}
Algorithm  EnergyState solves problem $\bar n, m~|~|~MU$ in $O(mn!Q_{\max}^n)$ time by either finding the problem to be infeasible or finding an optimal solution if the problem is feasible. 
\end{proposition}
{\bf Proof:} It can be seen that in the above algorithm, at each location, we keep at most $Q_{\max}^n$ feasible states of energy levels, and for each such state, we calculate the feasible states for the next location by enumerate all $n!$ possible sequences of the vehicles. Thus, the computation time at each location is $O(n!Q_{\max}^n)$. Therefore, the overall computation time of the algorithm is bounded by $O(mn!Q_{\max}^n)$. Since $n$ is fixed, this computation time is pseudo-polynomial. \blot

\bigskip

\begin{corollary} \label{cor:ESfptas}
Similar to Algorithms~DP1{$_{\epsilon}$} for problem~$\bar n,m|~|~TU$ in 
Section~\ref{sec4.3}, we can describe a family of algorithms
for problem~$\bar n, m~|~|~MU$, based on Algorithm~EnergyState and parameter $\epsilon>0$, such that each of these algorithms runs in 
$O(m[(1 + 2m/\epsilon) \ln(Q_{\max})]^n n!)$ time,  and if the algorithms generate a feasible solution, then the solution is a  $(1+\epsilon)$-approximation solution for the problem.
\end{corollary}

\subsection{Problem with Arbitrary $n$ and $m$} \label{sec3.5}
By Proposition~\ref{prop:2.5-snpc}, it follows immediately that finding a feasible solution to problem $n,m|| MU$ is also strongly NP-hard even when $C_j=C, \delta_j=\delta, S_{j1}=S, \beta_j=\beta$. 

Furthermore, if the initial energy level of each vehicle,  $C_jS_{j1}$, for $j=1,\ldots, n$, is sufficiently large that any solution is feasible, this problem is still strongly NP-hard, even when $C_j=C, \delta_j=\delta, S_{j1}=S, \beta_j=\beta$. This can be proved using the proof of Proposition~\ref{prop:2.5-snpc} with the following modifications: make $C_j$ and $S_{j1}$   sufficiently large such that  any solution is feasible (e.g., $S_{j1}C_j\ge \bar L$, where $\bar{L}\ge \beta C + \sum_{i=1}^m D_i(1-\eta_1)\delta$), and set the threshold of the maximum energy usage to be $(u-1)X+1$. 

Following the above discussions, we have the following proposition.

\begin{proposition}  \label{prop:3.5snpc}
Finding a feasible solution to problem $n,m~|~C_j=C, \delta_j=\delta, S_{j1}=S, \beta_j=\beta~|~MU$ is strongly NP-hard. Also, problem $n,m~|~C_j=C, \delta_j=\delta, S_{j1}=S\ge \frac{\bar L}{C}, \beta_j=\beta ~|~MU$ is also strongly NP-hard.
\end{proposition}
Proof: The same proof as used in Proposition~\ref{prop:2.5-snpc} applies after modifying $C$ and $S_{j1}$ to be sufficiently large that any solution is feasible, and setting the threshold of the maximum energy usage to be $(u-1)X+1$. \blot

We propose a greedy heuristic to solve problem  $n,m\mid \delta_j=\delta,  S_{j1}C_j\ge \bar L\mid MU$, and analyze its worst-case performance.

\medskip

\noindent {\bf Algorithm HEU1}

\noindent {\bf Step 0:} Reindex the $m$ segments in non-increasing order of their distances, i.e., $D_1\ge D_2\ge \cdots \ge D_m$.  Denote the total energy used by vehicle $j$ after the first $h$ segments as $E_{jh}$, Set $E_{j0}=0$, for $j=1, \ldots, n$. Set the segment counter $k=1$.

\noindent {\bf Step 1:} Sort the vehicles in  non-decreasing order of $E_{j,k-1}$, breaking ties by increasing order of vehicle indices. Assign the $h$th vehicle  to the $h$th position of segment $k$, for $h=1, \ldots, n$. Let $E_{jk}=E_{j,k-1}+D_k\delta_j(1-\eta_{[j]})$, where $[j]$ is the position where vehicle $j$ is assigned to, for $j=1,\ldots, n$.

\noindent {\bf Step 2:} If $k=m$, stop and adjust the sequence of the segments in the solution such that the solution considers the segments in the original order.  Otherwise, update $k=k+1$ and go to Step 1.

\bigskip
\begin{proposition}  \label{prop:3.5-heuristic}
For problem $n,m\mid \delta_j=\delta,  S_{j1}C_j \ge \bar L \mid MU$, the worst-case performance ratio of Algorithm HEU1 is bounded by $1+\frac{\eta_n-\eta_1}{1-\eta_1}$.  \end{proposition}
\noindent {\bf Proof:} In the solution generated by the algorithm, denote the maximum and minimum energy usage of the vehicles after the first $k$ segments (as reindexed in Step~0) as $E_{\max,k}$ and $E_{\min,k}$, respectively, for $k=1,\ldots, n$. Thus, $E_{\max,k}= \max\{E_{jk}|j=1,\ldots,n\}$, and  $E_{\min,k} = \min\{E_{jk}|j=1,\ldots,n\}$. 

We first show that $E_{\max,k}-E_{\min,k}\le D_1\delta(\eta_n-\eta_1)$, for $k=1,\ldots, m$, by induction. For ease of presentation, define $\lambda_h=1-\eta_h$, for $h=1,\ldots,n$. 

Observe that $E_{\max,1}-E_{\min,1}=D_1\delta\lambda_1- D_1\delta\lambda_n= D_1\delta(\eta_n-\eta_1)$. Suppose that $E_{\max,h}-E_{\min,h} \leq D_1\delta(\eta_n-\eta_1)$, for $h=1,\ldots,k$. We need to prove that $E_{\max,k+1}-E_{\min,k+1} \leq D_1\delta(\eta_n-\eta_1)$. Let $u$ and $v$ be the vehicles that have the maximum and minimum total usage, respectively, after the first $k+1$ segments. Thus, $E_{\max,k+1}-E_{\min,k+1} = E_{u,k+1}-E_{v,k+1}$. There are two cases to consider as follows.

\noindent Case 1: If $E_{uk} \le E_{vk}$, then we have 
\vspace{-0.15in}
\begin{eqnarray*}
E_{u,k+1}-E_{v,k+1} &=& (E_{uk}+D_{k+1}\delta\lambda_{i_u}) -(E_{vk}+D_{k+1}\delta\lambda_{i_v}) \\
&=& (E_{uk}-E_{vk}) +D_{k+1}\delta(\lambda_{i_u}-\lambda_{iv}) \\
&\le &  D_{k+1}\delta(\lambda_{i_u}-\lambda_{iv}) \le D_1\delta(\eta_n-\eta_1),
\end{eqnarray*}
where $i_u$ and $i_v$ are the positions where vehicles $u$ and $v$ are assigned to in segment $k+1$, respectively. 

\noindent Case 2: If $E_{uk}>E_{vk}$, then from Step~1 of the algorithm, in segment $k+1$, the position of vehicle $u$, say $i_u$, is later than the position of vehicle $v$, say $i_v$, that is, $i_u>i_v$, which implies that $\lambda_{i_u}\le\lambda_{i_v}$. Thus,
\vspace{-0.15in}
\begin{eqnarray*}
E_{u,k+1}-E_{v,k+1} &=& (E_{uk}+D_{k+1}\delta\lambda_{i_u}) -(E_{vk}+D_{k+1}\delta\lambda_{i_v}) \\
&=& (E_{uk}-E_{vk}) +D_{k+1}\delta(\lambda_{i_u}-\lambda_{iv}) \\
&\le &  E_{uk}-E_{vk} \le D_1\delta(\eta_n-\eta_1).
\end{eqnarray*}

This completes the induction proof. Hence $E_{\max,k}-E_{\min,k}\le D_1\delta(\eta_n-\eta_1)$, for $k=1,\ldots, n$. 

Since the vehicles have identical usage rates $\delta_j$, the total energy usage  by all the vehicles together in any solution is the same, which is $E=\sum_{i=1}^m D_i\delta(\lambda_1+\ldots+\lambda_n)$. Clearly, the optimal objective value satisfies $Z^*\ge E/n$, and $Z^*\ge D_1\delta(1-\eta_1)$. In the solution generated by the heuristic, we have $E_{\min,m}\le E/n$. Thus, 
\begin{equation}
Z^*\ge \max\{E_{\min,m}, ~D_1\delta(1-\eta_1)\}. \label{Z*}
\end{equation}
Let $Z^H$ be the objective value of the solution generated by the heuristic. Thus, $Z^H=E_{\max,m}$. By the induction result, $Z^H\le E_{\min,m}+D_1\delta(\eta_n-\eta_1)$. Therefore, along with (\ref{Z*}), we have
\[
\frac{Z^H}{Z^*} \le \frac{E_{\min,m}}{{Z^*}}  +\frac{D_1\delta(\eta_n-\eta_1)}{Z^*} \le 1+\frac{\eta_n-\eta_1}{1-\eta_1}. 
 \qquad  \blot \]

\medskip

Since $0\le \eta_1\le \eta_n<1$, $0< \frac{\eta_n-\eta_1}{1-\eta_1}\le 1$. Thus, with any values of $\eta_1$ and $\eta_n$, the worst-case performance ratio of the heuristic is always bounded by 2. When $\eta_n-\eta_1\rightarrow 0$, the worst-case performance ratio approaches 1, i.e., the solution generated by this heuristic is asymptotically optimal. The following result establishes a lower bound on the performance ratio.
\begin{lemma}
  The worst-case performance ratio of Algorithm HEU1 is between $\frac{4}{3}$ and 2.   
\end{lemma} 
\noindent {\bf Proof:} We construct an instance for which the performance ratio of our heuristic can be as large as $\frac{4}{3}$. Consider 3 vehicles and 3 segments (i.e., $n=m=3$), where $D_1=D, D_2=D-\epsilon, D_3=D-2\epsilon$ for a positive integer $D$, $\delta=1$, and $0\le \eta_1<\eta_2<\eta_3\le 1$ such that $\eta_2=\frac{1}{2}(\eta_1+\eta_3)$. Applying this heuristic, we obtain the following solution:
\vspace{-0.1in}
\begin{itemize}
   \item In segment 1, vehicle $j$ is assigned to position $j$, for $j=1,2,3$, which leads to 
   \vspace{-0.1in}
   \begin{itemize} 
        \item $E_{11}=D(1-\eta_1)$, 
        \item $E_{21}=D(1-\eta_2)$, 
        \item $E_{31}=D(1-\eta_3)$. ~~~Hence, $E_{11}>E_{21}>E_{31}.$
   \end{itemize}
   \vspace{-0.15in}
   \item In segment 2, vehicle $j$ is assigned to position $4-j$, for $j=1,2,3$, which leads to 
   \vspace{-0.1in}
   \begin{itemize} 
   \item $E_{12}=D(1-\eta_1) + (D-\epsilon)(1-\eta_3)$, 
   \item $E_{22}=D(1-\eta_2)+(D-\epsilon)(1-\eta_2)$, 
   \item $E_{32}=D(1-\eta_3)+(D-\epsilon)(1-\eta_1)$. ~~~~Hence, $E_{12}>E_{22}>E_{32}.$
   \end{itemize}
     \vspace{-0.15in}
     \item In segment 3, vehicle $j$ is again assigned to position $4-j$, for $j=1,2,3$, which leads to 
     \vspace{-0.1in}
     \begin{itemize}
        \item $E_{13}=D(1-\eta_1) + (D-\epsilon)(1-\eta_3) + (D-2\epsilon)(1-\eta_3)$, 
        \item $E_{23}=D(1-\eta_2)+(D-\epsilon)(1-\eta_2)+(D-2\epsilon)(1-\eta_2)$,
        \item $E_{33}=D(1-\eta_3)+(D-\epsilon)(1-\eta_1)+(D-2\epsilon)(1-\eta_1)$.
       \end{itemize}
\end{itemize}
It is evident that when $\epsilon\rightarrow 0$, $E_{33}>E_{23}>E_{13}$, and hence the objective value of this solution is $Z^H=E_{33} = D[3-(2\eta_1+\eta_3)]$, as $\epsilon\rightarrow 0$. An optimal solution assigns each vehicle to a different position in each segment, i.e., vehicle 1 is assigned to positions 1, 2, 3 in the three segments, vehicle 2 is assigned to positions 2, 3, 1 in the three segments, and vehicle 3 is assigned to positions 3, 1, 2 in the three segments. This gives the optimal objective value $Z^*=D[3-(\eta_1+\eta_2+\eta_3)]$, as $\epsilon\rightarrow 0$. Thus,
 \vspace{-0.1in}
\begin{equation}
\frac{Z^H}{Z^*} = \frac{3-(2\eta_1+\eta_3)}{3-(\eta_1+\eta_2+\eta_3)} = \frac{3-(2\eta_1+\eta_3)}{3-(1.5\eta_1+1.5\eta_3)} = 1 + \frac{\eta_3-\eta_1}{6-3(\eta_1+\eta_3)}. \label{instance-ratio}
\end{equation}
Since $\eta_3\le 1$, we can show that the right-hand-side of (\ref{instance-ratio}) is always less than or equal to $\frac{4}{3}$, and when $\eta_3\rightarrow 1$, it approaches $\frac{4}{3}$, for any value of $\eta_1<\eta_3$. This shows that  the worst-case performance ratio
of heuristic HEU1 is between $\frac{4}{3}$ and 2. \blot

\medskip

Next, we propose a greedy heuristic to solve the more general problem  $n,m\mid S_{j1}C_j\ge \bar L \mid MU$, where the usage rates $\delta_j$'s are not assumed to be identical, and analyze its worst-case performance.

\medskip

\noindent {\bf Algorithm HEU2}

\noindent {\bf Step 0:} As in Step~0 of Algorithm HEU1.

\noindent {\bf Step 1:} Assign the vehicles to the $n$ positions in segment $k$ to minimize the maximum energy usage of the vehicles over the first $k$ segments  by the following bottleneck maximum matching procedures. 
 
\noindent {\bf Step 1.1:} First, create a set $G$ of  all candidate values of the maximum energy usage after the first $k$ segments as follows: for each vehicle $j$, for $j=1, \ldots, n$, assign it to every possible position $i$ in segment $k$ and add the resulting energy usage $E_{j,k-1}+D_{k}\delta_j(1-\eta_i)$ to set $G$, for $i=1,\ldots,n$.

\noindent {\bf Step 1.2:} Suppose that $G$ contains  $u$ distinct numbers. Sort these distinct numbers in increasing order and denote them as $g_1<g_2<\cdots<g_u$. Perform a binary search over $G$ as follows. Let $l=1, r=u$ and $t=\lfloor (l+r)/2\rfloor$. Given the target value $g_t\in G$, execute Steps 1.3 and 1.4.

\noindent {\bf Step 1.3:} Create a bipartite graph consisting of the $n$ vehicles on the left side and the $n$ positions in segment $k$ on the right side, and the arc connecting  vehicle $j$ with each position $i$, satisfying $E_{j,k-1}+D_{k}\delta_j(1-\eta_i)\le g_t$, for $i,j=1, \ldots,n$.

\noindent {\bf Step 1.4:}  Find a maximum matching  in the constructed bipartite graph. If the maximum matching found contains fewer than $n$ arcs, update $l=t+1$ and $t=\lfloor (l+u)/2\rfloor$, go to Step 1.3.  Otherwise, if $t>l$, then update $u=t-1$ and $t=\lfloor (l+u)/2\rfloor$, go to Step 1.3, and otherwise, a solution is found for segment $k$, which is to assign vehicle $j$ to position $h$, where $(j,h)$ is the arc contained in the last maximum matching found that connects vehicle $j$, and let $E_{jk}=E_{j,k-1}+D_k\delta_j(1-\eta_h)$, for $j=1,\ldots, n$.

\noindent {\bf Step 2:} If $k=n$, stop and adjust the sequence of the segments in the solution such that the solution considers the segments in their original order.  Otherwise, update $k=k+1$ and go to Step 1.

\medskip

\begin{proposition}  \label{prop:3.5-heu2}
For problem $n,m\mid S_{j1}C_j\ge \bar L \mid MU$, the worst-case performance ratio of Algorithm HEU2 is bounded by $\frac{\max_{1\le j\le n}\{\delta_j(1-\eta_j)\}}{\min_{1\le j\le n} \{\delta_j(1-\eta_j)\}}$, where the $\delta_j$ values are reindexed such that $\delta_1\le \cdots \le \delta_n$.  \end{proposition}
\noindent {\bf Proof:} Reindex the vehicles such that $\delta_1\le \cdots \le \delta_n$. Define $A_{\max}=\max_{1\le j\le n}\{\delta_j(1-\eta_j)\}$ and $A_{\min}=\min_{1\le j\le n}\{\delta_j(1-\eta_j)\}$. By Proposition~\ref{prop-rule4},  assigning vehicle $j$ to position $j$ in every segment $i$, for $i,j=1,\ldots,m$ minimizes the total energy usage of all the vehicles. Thus, the minimum total energy usage of all the vehicles, denoted as $F^*$, is
 \vspace{-0.15in}
 \[
F^*=\sum_{i=1}^m \sum_{j=1}^n D_i\delta_j(1-\eta_j)=\left(\sum_{i=1}^m D_i\right)\left(\sum_{j=1}^n\delta_j(1-\eta_j)\right).
\]
This implies that the minimum possible maximum energy usage by any individual vehicle, denoted as $Z^*$, satisfies
\begin{equation} \label{eq-heu2-1}
Z^*\ge F^*/n = \left(\sum_{i=1}^m D_i\right)\frac{\sum_{j=1}^n\delta_j(1-\eta_j)}{n}\ge \left(\sum_{i=1}^m D_i\right)A_{\min}.
\end{equation}
In the solution generated by the algorithm, denote the maximum energy usage of the vehicles after the first $k$ segments as $E_{\max,k}$, for $k=1, \ldots, m$. Clearly, 
 \vspace{-0.1in}
\begin{equation} \label{eq-heu2-2}
E_{\max,1} = D_1A_{\max}.
\end{equation}
Now, consider any segment $k\ge 2$. Let $(k_1,\ldots, k_n)$ be the sequence of the positions assigned to the $n$ vehicles by the algorithm in segment $k$.  Since the algorithm considers all possible sequences of the positions, including $(1, \ldots, n)$, we have
 \vspace{-0.2in}
\begin{eqnarray}
E_{\max,k}&=&\max\{E_{j,k-1} + D_k\delta_j(1-\eta_{k_j}) \mid j=1,\ldots, n\} \nonumber \\
& \le & \max\{E_{j,k-1} + D_k\delta_j(1-\eta_{j}) \mid j=1,\ldots, n\} \nonumber \\
&\le & \max\{E_{j,k-1} \mid j=1, \ldots, n\} + D_k\max\{\delta_j(1-\eta_{j}) \mid j=1,\ldots, n\} \nonumber \\
&=& E_{\max,k-1} + D_kA_{\max}. \label{eq-heu2-3}
\end{eqnarray}
By (\ref{eq-heu2-2}) and (\ref{eq-heu2-3}),  the objective of the solution generated by the algorithm, denoted as $Z^H$, satisfies the following:
 \vspace{-0.1in}
\[
Z^H = E_{\max,m} \le \left(\sum_{i=1}^mD_i\right)A_{\max}.
\]
Therefore, along with (\ref{eq-heu2-1}), we have $Z^H/Z^*\leq \frac{A_{\max}}{A_{\min}}
= \frac{\max_{1\le j\le n}\{\delta_j(1-\eta_j)\}}{\min_{1\le j\le n} \{\delta_j(1-\eta_j)\}}$. \blot

\bigskip

\begin{remark}  \label{rem2}
The ratio bound established in Proposition~\ref{prop:3.5-heu2} is small in typical platoons, especially those that include vehicles of the same general type. When all vehicles are HGVs, a typical ratio of $\delta_{\max} / \delta_{\min}$ is 1.15; for LGVs or delivery vans, it is 1.13; and for passenger vehicles it is 1.25, whereas the maximum ratio of energy saving due to position within the platoon, $\eta_n/\eta_1$, is typically 1.18, 1.15, and 1.33 for the same three vehicle types, respectively (McAuliffe et al 2018, Kaluva et al. 2020, Yang et al. 2019). Since,
for vehicle $j$, $A_j$ is calculated by combining these two parameters in antithetical ordering, a typical platoon composed of vehicles of the same type can be expected to have $A_{\max} / A_{\min} \approx 1.05 - 1.10$. Therefore, solutions delivered by Algorithm~HEU2 are guaranteed to be close to optimal even in the worst case. 
\end{remark}

Finally, we conduct a computational experiment to evaluate the typical performance of Algorithms HEU1 and HEU2 for problems $n,m\mid \delta_j=\delta, S_{j1}C_j\ge \bar L \mid MU$ and $n,m\mid S_{j1}C_j\ge \bar L \mid MU$, respectively, using randomly generated instances with parameter ranges that closely reflect real-world situations. For both problems, test instances are generated with $n\in \{10, 15\}$ and $m\in \{6, 10\}$, $D_j\in U\{30, 50\}$, and $\delta_j$ and $\eta_j$ are problem dependent as follows: 
\begin{itemize}
    \item For the first problem, $\delta_1=\cdots =\delta_n=1$; for the second problem, two cases of $\delta_j$ parameters: either $\delta_j\in U[0.8, 1]$ or $\delta_j\in U[0.9, 1]$. 
    \item For both problems, $\eta_j$ values are set as follows: (i) for instances with $n=10$, $\eta_1=0.04$, $\eta_j=\eta_{j-1}+\frac{11-j}{45}\times 0.14$, for $j=2, \ldots, 10$, which gives $\eta_{10}=0.18$, and (ii) for instances with $n=15$, $\eta_1=0.04$, $\eta_j=\eta_{j-1}+\frac{16-j}{105}\times 0.14$, for $j=2, \ldots, 15$, which also gives $\eta_{15}=0.18$. In both cases, the $\eta_j$ values are increasing from 0.04 to 0.18, but the increment from $\eta_{j-1}$ to $\eta_{j}$ is decreasing, as $j$ goes from 1 to $n$, reflecting diminishing energy savings towards the end of larger platoons. 
\end{itemize}
For each combination of $n$ and $m$ for the first problem, and each combination of $n$, $m$ and $\delta_j$ range for the second problem, 20 random instances are generated and solved by the corresponding heuristic. They are also formulated as an integer program and solved to optimality using CPLEX MIP Solver. The MIP Solver can take excessive time to solve many instances of the first problem with $n=15$ and all instances of the second problem, so we set a time limit of 600 seconds for the solver and use the best lower bound found within this time limit in place of the optimal objective value. As a result, the numbers reported underestimate heuristic performance. For each instance, the objective value of the heuristic solution is compared to the optimal objective value or the best lower bound found. The average and maximum relative $\%$ gaps across the 20 instances are reported in the following tables. 

\begin{table}[htbp]
\scriptsize
    \begin{minipage}[t]{0.45\textwidth}
  \centering
  \caption{HEU1 Computational results}
  \label{tab:identical_vehicles}
  \begin{tabular}{ccccc}
    \toprule
    $n$ & $m$ & Avg Gap (\%) & Max Gap (\%) \\
    \midrule
    10 & 6 & 0.80 & 1.11  \\
    10 & 10 & 0.29 & 0.45  \\
    \midrule
    15 & 6 & 0.70 & 1.05  \\
    15 & 10 & 0.27 & 0.66  \\
    \bottomrule
  \end{tabular}
   \end{minipage}
   \begin{minipage}[t]{0.45\textwidth}
 \centering
  \caption{HEU2 Computational results}
  \label{tab:general_delta}
  \begin{tabular}{ccccc}
    \toprule
    $n$ & $m$ & $\delta_j$ range & Avg Gap (\%) & Max Gap (\%)  \\
    \midrule
    10 & 6 & [0.8, 1.0] & 0.03 & 0.29  \\
    10 & 6 & [0.9, 1.0] & 0.44 & 0.92  \\
    10 & 10 & [0.8, 1.0] & 0.04 & 0.25  \\
    10 & 10 & [0.9, 1.0] & 0.28 & 0.55  \\
    \midrule
    15 & 6 & [0.8, 1.0] & 0.02 & 0.42   \\
    15 & 6 & [0.9, 1.0] & 0.53 & 0.74  \\
    15 & 10 & [0.8, 1.0] & 0.00 & 0.06  \\
    15 & 10 & [0.9, 1.0] & 0.27 & 0.38  \\
    \bottomrule
  \end{tabular}
   \end{minipage}
\end{table}

These test results show that both heuristics perform extremely consistently and extremely well with an optimality gap always less than 1.11\% for HEU1 over 800 random instances, and less than 0.53\% for HEU2 over 1600 random instances. Encouragingly, all the results show an improvement in heuristic performance as the number of segments, $m$, increases. For HEU2, reducing the range of $\delta$ increases the performance gap due to greater vehicle similarity.

\section{Extensions}  \label{sec:extensions}

In this section, we consider two extensions to the problems considered in the previous two sections. Both extensions are motivated by  commonly observed vehicle platooning practice. The first extension is that, in a platoon, only the first few positions have different usage rate reduction parameters (i.e., different $\eta_j$ values), while all the remaining positions have very similar usage rate reduction parameters such that they can be assumed to be identical. This extension is motivated by the findings from some studies (e.g., Zabat et al. 1995) that for vehicles positioned fourth or beyond in a platoon, the usage rate reduction tends to stabilize and hence the usage rate reduction parameters for the positions starting from the fourth one are similar. The second extension is the constraint that there are limited position changes between vehicles when resequencing them. This situation is motivated by limitations of vehicle-to-vehicle communication and safety concerns, especially during on-road resequencing.

\subsection{Only First Few Positions Matter}

Coppola et al. (2022), in a study of position-dependent energy savings in platoons, observe that the increase in savings between consecutive vehicles diminishes towards the end of the platoon.
Motivated by this observation, and also by the difficulty of accurately estimating positional energy savings rates, we consider a setting where only the first few positions in a platoon have different usage rate reduction parameters, while the remaining positions all have the same usage rate reduction parameter. Specifically, there is a fixed, typically small $K$ (e.g., $K=4$) such that $\eta_1\le \cdots \le \eta_K = \eta_{K+1}=\cdots=\eta_{n}$. We use ``$(\eta_1, \ldots, \eta_K)"$ to represent this condition. 

We have the following result for both the problems with $TU$ and $MU$ objectives in this case. 
\begin{proposition} \label{prop-6.2-1}
Both problems $n,\bar m\mid (\eta_1,\ldots,\eta_K) \mid TU$ and $n,\bar m\mid (\eta_1,\ldots,\eta_K) \mid MU$, can be solved in $O(mn^{Km+1})$ time.
\end{proposition}
{\bf Proof:} We first find all feasible solutions for segment 1. Each position $j$, for $j=1,\ldots, K$, can be assigned to exactly one vehicle. Thus, for these $K$ positions together, there are a total of $n(n-1)\cdots(n-K+1)$ possible solutions. If a vehicle is not assigned to one of the first $K$ positions, it is assigned to an arbitrary one among  the last $n-K$ positions because the sequence of vehicles in the last $n-K$ positions does not affect energy usage. Thus, there are $n(n-1)\cdots(n-K+1)$ possible solutions in segment 1. This same logic applies to any subsequent segment, and hence, for each possible solution covering the first $h$ segments, for $h=1, \ldots, m-1$, this solution can be extended to include the next segment $h+1$ in $n(n-1)\cdots(n-K+1)$ possible ways. Therefore, there are a total of $[n(n-1)\cdots(n-K+1)]^m\le n^{Km}$ possible complete solutions over all the segments. 

Observe that not all the possible solutions found above are feasible. Hence, it is necessary to check the feasibility and calculate the objective value of each possible solution, which requires $O(mn)$ time. After that, an optimal solution can be found by comparing the energy usage of the remaining feasible solutions. The overall computational time is thus bounded by $O(mn^{Km+1})$. This time is polynomial since both $K$ and $m$ are fixed. \blot

\medskip

We note that when $m$ is arbitrary, these problems generalize to the problems $\bar{n},m\mid \mid TU$ and $\bar{n},m\mid \mid MU$, respectively, which are considered in Sections~\ref{sec4.3} and ~\ref{sec5.3}, respectively, and hence are  ordinarily NP-hard. 

Motivated by Proposition~\ref{prop-6.2-1}, we next investigate the following question: given a problem where the usage rate reduction parameters (i.e., the $\eta_j$ values) do not satisfy the condition stated in this proposition, if we approximate some of these parameters such that the condition is satisfied (and hence the approximate problem can be solved in polynomial time), how good is the resulting solution?

\begin{proposition} \label{prop-6.2-2}
Given either one of the problems $n,\bar m\mid \mid TU$ and $n,\bar m\mid \mid MU$
where the condition ``$(\eta_1, \ldots, \eta_K)$" stated in Proposition~\ref{prop-6.2-1} does not hold, suppose we approximate the problem by assuming that $\eta_{K+1}= \cdots= \eta_n= \eta_K$, for a fixed $K$, such that this condition holds. We then solve the resulting approximate problem $n,\bar m\mid (\eta_1, \ldots, \eta_K) \mid TU$ or $n,\bar m\mid (\eta_1, \ldots, \eta_K) \mid MU$. If the approximate problem is feasible, then the  solution found for this problem is also feasible for the original problem, and furthermore, $\frac{\tilde{Z}}{Z^*}\leq 1 + \frac{\eta_n-\eta_K}{1-\eta_n}$, where  $\tilde{Z}$ and $Z^*$ are the optimal objective values of the  approximate problem and the original problem, respectively. 
\end{proposition}
{\bf Proof:} The approximate problem is identical to the original problem except that the usage rate reduction parameters for positions $K+1, \ldots, n$ are different. For ease of presentation, we denote the usage rate reduction parameters in the approximate problem as $\eta_1', \ldots, \eta_n'$, where $\eta'_j=\eta_j$ for $j=1,\ldots,K$, and $\eta'_{K+1}=\cdots=\eta'_{n}=\eta_K$, for a given integer $K$. Since $\eta_1\le \cdots \le \eta_n$, we have $\eta'_j\le \eta_j$ for $j=1,\ldots, n$. 
These relationships imply that 
 \vspace{-0.15in}
\begin{eqnarray*}
1-\eta_{j}' & = & 1 - \eta_{j}, ~~\mbox{for $j = 1, \ldots, K$}, \\
1-\eta_j\le 1-\eta_{j}' & = & 1-\eta_K \le  (1 - \eta_{K})\frac{1-\eta_{j}}{1-\eta_n}=  (1 - \eta_{j})\frac{1-\eta_K}{1-\eta_n}, ~~\mbox{for $j = K+1 \ldots, n$.} 
\end{eqnarray*}
This, together with the fact that $\eta_K\le \eta_n$,  implies that
 \vspace{-0.1in}
\begin{equation} \label{prop-6.2-2-eq1}
1-\eta_j\le 1-\eta_{j}' \le (1 - \eta_{j})\frac{1-\eta_K}{1-\eta_n}, ~~\mbox{for $j = 1,\ldots, n$.} 
\end{equation}

If the approximate problem is feasible, we denote the optimal solution found for this problem as $\pi_0$. Let $\pi^*$ be any given optimal solution for the original problem. In a given solution $\pi$ (which can be for the approximate and/or the original problem), let the position of vehicle $j$ in segment $i$ be denoted as $j_i(\pi)$, for $i=1,\ldots, m$ and $j=1,\ldots, n$. In addition, we use $U_j^0(\pi)$ and $U_j(\pi)$ to denote the total energy usage by vehicle $j$ under a given  solution $\pi$ for the approximate problem and for the original problem, respectively, for $j=1,\ldots,n$.  Then, by (\ref{prop-6.2-2-eq1}), we have
 \vspace{-0.1in}
\begin{equation} \label{prop-6.2-2-eq2}
U^0_j(\pi_0) = \sum_{i=1}^m D_i\delta_j(1-\eta'_{j_i(\pi_0)})\ge \sum_{i=1}^m D_i\delta_j(1-\eta_{j_i(\pi_0)}) = U_j(\pi_0), ~~\mbox{for $j=1,\ldots, n$.}
\end{equation}
The feasibility of $\pi_0$ for the approximate problem implies that $S_{j1} - U_j^0(\pi_0)/C_j \ge \beta_j$ for $j=1,\ldots, n$. This, together with (\ref{prop-6.2-2-eq2}), implies that $S_{j1} - U_j(\pi_0)/C_j \ge \beta_j$ for $j=1,\ldots, n$. Thus, $\pi_0$ is also feasible for the original problem. 

Let $Z^0(\pi)$ and $Z(\pi)$ denote the objective value of solution $\pi$ for the approximate problem, and for the original problem, respectively.  Since $\pi_0$ is optimal for the approximate problem, we have $Z^0(\pi_0) \le Z^0(\pi^*)$. Further, since 
$\eta_j'\le \eta_j$ for $j=1,\ldots, n$, we have $Z^0(\pi_0)\ge Z(\pi_0)$. Thus, 
 \vspace{-0.1in}
\begin{equation} \label{prop-6.2-2-eq3}
Z(\pi_0)\le Z^0(\pi_0) \le Z^0(\pi^*).
\end{equation}
By (\ref{prop-6.2-2-eq1}), we have
 \vspace{-0.15in}
\begin{eqnarray} \label{prop-6.2-2-eq4}
U^0_j(\pi^*) & =& \sum_{i=1}^m D_i\delta_j(1-\eta'_{j_i(\pi^*)}) \nonumber \\
& \le & \frac{1-\eta_K}{1-\eta_n}\sum_{i=1}^m D_i\delta_j(1-\eta_{j_i(\pi^*)}) \nonumber \\
& = & \frac{1-\eta_K}{1-\eta_n}U_j(\pi^*), ~~~~~~\mbox{for $j=1,\ldots, n$.}
\end{eqnarray}
Since the objective functions of the problems covered in the proposition are either the summation or the maximization of the individual vehicle usages, (\ref{prop-6.2-2-eq4}) implies that
$Z^0(\pi^*) \le \frac{1-\eta_K}{1-\eta_n} Z(\pi^*)$. This, together with (\ref{prop-6.2-2-eq3}), further implies that
 \vspace{-0.1in}
\[
Z(\pi_0) \le \frac{1-\eta_K}{1-\eta_n}Z(\pi^*) = \left[1 + \frac{\eta_n-\eta_K}{1-\eta_n}\right]Z(\pi^*). \qquad \blot
\]

\medskip

\subsection{Limited position changes} \label{sec3.6}

We consider an extension of problem $n,m\mid \mid MU$ where, given the vehicle position sequence for segment 1, the decision maker needs to determine a position sequence of the vehicles in each of the following segments subject to the constraint that at most $K$ position changes are allowed for each vehicle between consecutive segments, where $K \ge 1$ is a known constant. For ease of presentation, we denote this problem $n,m\mid \pi_1, pc\le K \mid MU$, where $\pi_1$ is the given platoon sequence of the vehicles for the first segment.
This requirement that each vehicle can have a limited position change from one segment to the next is widely used by vehicle platoons due to (a) spatial constraints and safety concerns especially for on-road resequencing, (b) the exponential complexity of full resequencing, (c) the risk of string instability, and (d) technical issues in real-world wireless environments including signal shading, network interference, and packet drops (Tian et al. 2005). The vehicles are assumed to have general parameters. We assume without loss of generality that the given sequence of vehicles in segment~1 is $\pi_1=(1,\ldots ,n)$.

The constraint on vehicle position change can be formalized as follows. The set of allowed positions for a vehicle at the next segment, if it is at position $h$ at the current segment, is denoted by $A(h) = \{u, u+1, \ldots, h, h+1, \ldots, v\}$ with $u = \max\{1, h - K\}$ and  $v = \min\{n, h + K\}$, where $|A(h)| \leq 2K+1$. 

We define a {\it position path} of a vehicle as an $m$-dimensional vector specifying the positions of the vehicle in the initial platoon and the platoon in each of the $m-1$ segments that follow. A position path is feasible if the position change from each segment to the next is no more than $K$.

\noindent {\bf Example~1.} Consider $K=2$,  $m=3$ and $n \ge 5$. For vehicle~1, which is at position 1 in the first segment, the feasible position paths for this vehicle are the following vectors: 
 \vspace{-0.1in}
\begin{eqnarray*}
&& (1, 1, 1), (1, 1, 2), (1, 1, 3); \\
&&(1, 2, 1), (1, 2, 2), (1, 2, 3), (1, 2, 4); \\
&&(1, 3, 1), (1, 3, 2), (1, 3, 3), (1, 3, 4), (1, 3, 5).
\end{eqnarray*}

Observe that, for any vehicle $j$, there are at most $(2K+1)^{m-1}$ feasible position paths. The total energy usage $U_j$ as well as the final SoC of the vehicle $S_{j,m+1} = S_{j1}-U_j/C_j$ for each position path can be calculated in $O(m)$ time. An energy usage value of a vehicle is feasible if the corresponding final SoC of the vehicle is at least $\beta_j$. Thus, all feasible values of energy usage of a vehicle can be calculated in $O(m(2K+1)^{m-1})$ time. Considering all the $n$ vehicles together, there are thus at most $O(nm(2K+1)^{m-1})$ distinct feasible energy usage values of the vehicles, and calculating these values requires at most   $O(nm^2(2K+1)^{m-1})$ time. Let $\Gamma$ denote the set of all the distinct feasible energy usage values of the vehicles. 
These feasible position paths of different vehicles may conflict if they share the same position at the same segment. We define a  feasible solution for a given subset $W$ of the vehicles to be a solution consisting of non-conflicting feasible position paths for all vehicles in $W$. 

We can solve the problem using binary search over the set $\Gamma$. 
At each iteration of the binary search, given a specific threshold cost of energy usage $T_0\in \Gamma$, our goal is find a feasible solution to the problem. Define $L_j = \max\{1, j-2(m-1)K\}$. We first show the following property.

\begin{lemma} \label{lemma-sec6.2}
When searching for all feasible position paths for a vehicle $j$, among the vehicles with an initial position before $j$, we can ignore the feasible position paths of  vehicles $1, 2, \ldots, L_j-1$, but cannot ignore feasible position paths of vehicles  $L_j, L_j+1, \ldots, j-1$. 
\end{lemma}
{\bf Proof:} Consider any vehicle $j$. The smallest position index this vehicle can occupy in segment $m$ is position $\max\{1, j - (m-1)K\}$, and the largest position index it can occupy in segment $m$ is position $\min\{n, j+(m-1)K\}$. This means that for any two vehicles $j_1$ and $j_2$, if $|j_2 - j_1|\geq 2(m-1)K+1$, then none of $j_1$'s feasible position paths overlap with any of $j_2$'s feasible position paths. Define $L_j = \max\{1, j-2(m-1)K\}$. This implies that, when searching for all feasible position paths for a vehicle $j$, among the vehicles with an initial position before $j$, we can ignore the feasible position paths of  vehicles $1, 2, \ldots, L_j-1$, but cannot ignore feasible position paths of vehicles  $L_j, L_j+1, \ldots, j-1$. \blot

\medskip

In the  algorithm described below, each iteration either generates all feasible solutions with the maximum energy usage of the vehicles no less than the given threshold value $T_0$, or shows that no such feasible solutions exist. It first generates all the feasible position paths for vehicle~1. Then in each iteration $j$, for $j=2, \ldots, n$, it expands the feasible solutions generated so far for the first $j-1$ vehicles by adding all the feasible position paths for vehicle $j$. By Lemma~\ref{lemma-sec6.2}, this only requires checking the feasible position paths of vehicles  $L_j, L_j+1, \ldots, j-1$. 

Before describing the algorithm, we define the following notation, for $j=1, \ldots, n$: \\
(a) $V_j = \{L_j, L_j+1, \ldots, j-1\}$,  the set of the only vehicles with an index smaller than $j$ that need to be considered when finding feasible position paths for vehicle $j$. The above discussion shows that $|V_j|\leq 2(m-1)K$.   \\
(b) $F_j$ = the set of all the feasible solutions for the vehicles in $V_j$ where the total energy usage of each vehicle does not exceed the threshold $T_0$. Since, as discussed above, for any vehicle, there are at most $(2K+1)^{m-1}$ feasible position paths, $|F_j| \le [(2K+1)^{(m-1)}]^{(2(m-1)K)}=O((2K)^{(2m^2K)})$.

We now describe the algorithm in detail. 

\medskip

\noindent {\bf Algorithm~LimitedPC}

\noindent {\bf Step 0:} Sort the elements in $\Gamma$ in non-decreasing order, and denote them as $\gamma_1, \ldots, \gamma_{|\Gamma|}$. Let $l=1$ and $r=|\Gamma|$. Set the target $t=\lfloor |\Gamma|/2\rfloor$. Let $T_0=\gamma_t$.

\noindent {\bf Step 1:} Generate all feasible position paths for vehicle~1, where each position path is an $m$-dimensional vector indicating the position of vehicle~1 in the $m$ segments. Keep the position paths with a total energy usage no more than $T_0$ and a final SoC no less than $\beta_1$ as a set $F_1$, where $|F_1| \leq (K+1)^{m-1}$.  Let $j=2$.

\noindent {\bf Step 2:} Consider all possible ways to add vehicle $j$ to every feasible solution for the first $j-1$ vehicles. For any solution in $F_j$, insert vehicle $j$ into this solution in any possible way, subject to the constraint that at most $K$ position changes are allowed from one phase to the next.  This is implemented as follows:

\noindent {\bf Step 2.1:} Use an 
$m|V_j|$-dimensional vector to represent each feasible solution in $F_j$, where the first $m$ dimensions represent the feasible path of vehicle $L_j$ in this solution (i.e., the $m$ positions of vehicle $L_j$ across the $m$ segments, respectively), the next $m$ dimensions represent that of vehicle $L_j+1, \ldots, j-2$, and the last $m$ dimensions represent that of vehicle $j-1$. 

\noindent {\bf Step 2.2:} For any vector $f\in F_j$, insert vehicle $j$ into $f$ by using the feasible positions in each segment that do not appear in $f$ subject to the constraint of at most $K$ position changes. Keep the resulting new vectors where the total energy usage of vehicle $j$ is no more than $T_0$ and its final SoC value of is at least $\beta_j$. 

\noindent {\bf Step 2.3:} If in the above procedure, for every vector  $f\in F_j$, no feasible position path is generated for vehicle $j$, then there is no feasible solution for the overall problem with the given threshold $T_0$ in the binary search, then go to Step~3. Otherwise,  create the set $F_{j+1}$ of all feasible solutions for the vehicles in $V_{j+1}$ by (i) first generating new vectors by inserting $j$ into the vectors in $F_j$, as described above, and (ii) removing the first $m$ dimensions (i.e., those corresponding to vehicle $L_j$) from each newly generated vector, and the set of the resulting vectors becomes $F_{j+1}$ to be used in the next iteration. If $j=n$, go to Step~3; otherwise, let $j=j+1$ and go to Step 2.2.

\noindent {\bf Step 3:} If no feasible position path is generated for vehicle $j$, then if $t=r$, then the problem is infeasible; otherwise, update $l=t$ and $t=\lfloor (t+r)/2\rfloor$ and go to Step 1. If $j=n$, update $r=t$ and $t=l+\lfloor (t-l)/2\rfloor$ and go to Step 1. 

\bigskip

We now present our main result.
\begin{proposition}  \label{prop:swap}
Algorithm~LimitedPC solves problem $n,m\mid \pi_1, pc\le K\mid MU$ in \\
$O \bigl ((\log n + m\log k)n(2K)^{(2m^2K+m)} \bigr )$ time.
\end{proposition}
{\bf Proof:} The algorithm uses binary search to search over the set $\Gamma$, which includes all distinct possible values of the objective value of the problem. Considering a single iteration of the binary search with a threshold cost $T_0$, Algorithm~LimitedPC requires $n$ iterations, one for each vehicle. The algorithm first generates all the feasible position paths for vehicle 1. At each iteration $j$, for $j=2, \ldots, n$, it expands the feasible solutions generated so far for the first $j-1$ vehicles by adding all feasible position paths for vehicle $j$. By Lemma~\ref{lemma-sec6.2}, at iteration $j$, we only need to check the feasible position paths of vehicles  $L_j, L_j+1, \ldots, j-1$. Thus, the algorithm considers all possible solutions and finds a solution of cost at most $T_0$ if one exists. The value of $T_0$ is then adjusted using binary search, and terminates when the smallest such value has been found. 

Regarding the running time of the algorithm, the binary search is carried out over the set $\Gamma$, thus it involves at most $\log (|\Gamma|) = O(\log [nm(2K+1)^{m-1}]) = O(\log n +m\log K)$ binary search iterations. Each  iteration considers all vehicles $j=1, \ldots, n$ and for each vehicle $j$, considers all the vectors in $F_j$, and for each vector, considers all feasible position paths for the vehicle. Hence, each binary search iterations takes $O(n(2K)^{(2m^2K)}(2K)^m) = O(n(2K)^{(2m^2K+m)})$ time. 
Therefore, the overall time required by the overall algorithm is bounded by $O \bigl ((\log n + m\log K)n(2K)^{(2m^2K+m)} \bigr )$, which is bounded by $O(n\log n)$ when $m$ and $K$ are both fixed.  \blot

\section{Concluding Remarks} \label{sec:conclude}

Vehicle platooning offers substantial opportunities for reducing energy consumption, operating costs, and environmental impacts. Motivated by these benefits, this paper investigates the optimal sequencing and resequencing of vehicles within a platoon under two widely studied energy-related objectives and a broad range of practically relevant operating conditions. Our analysis provides a comprehensive characterization of these problems: for all but one problem case considered, we either develop an efficient optimal algorithm or establish formal computational intractability. For the intractable cases, we introduce the first algorithms with provable worst-case performance guarantees, and show computationally that their average performance is extremely close to optimal. We also develop an efficient solution method for settings in which vehicle position changes are restricted during on-road resequencing. Collectively, these results substantially advance the theoretical understanding of vehicle platoon sequencing and resequencing, an area that has received comparatively little rigorous analysis despite its practical importance.

The results of our work have several important implications for the management of vehicle platoons. In a number of practically relevant settings, our algorithms enable optimal resequencing decisions that can yield greater energy savings than existing heuristic approaches. Where optimal solutions are computationally difficult to obtain, the proposed approximation and heuristic procedures provide managers with simple and highly effective decision-support tools whose performance can be quantified in advance. The results also highlight the value of accurately estimating position-dependent energy savings, since improved information can be translated directly into better sequencing decisions. 
These insights can inform both platoon formation strategies and the design of future intelligent transportation systems.

Several directions for future research remain. Among the problems considered in this paper, the complexity of sequencing vehicles across three road segments remains unresolved; this problem is equivalent to a special case of Numerical 3-Dimensional Matching in which the three sets are identical. Beyond this open question, many important platooning decisions remain largely unexplored. Promising research directions include the integration of sequencing and charging decisions, dynamic resequencing when vehicles enter or leave a platoon during a journey, and the allocation of heterogeneous vehicles across multiple platoons. As advances in autonomous driving, vehicle-to-vehicle communication, and intelligent highway infrastructure enable the safe operation of increasingly larger platoons, these optimization challenges will become both more difficult and more consequential. We hope that the theoretical foundations established in this paper will stimulate further research on these problems and contribute to realizing the full economic and environmental potential of large-scale vehicle platooning systems.

\section*{References}  \label{sec:references}
\vspace{-0.15in}
\baselineskip=15pt
\parskip=12pt

\begin{hangref}

\item Barhoumi, O., G. Farhani, T. Rahman, M.H. Zaki, S. Tahar, F. Araji. 2025. Fuel consumption in platoons: A literature review. 
Available at: https://arxiv.org/html/2508.10891v1.

\item Bhoopalam, A.K., N. Agatz, R. Zuidwijk. 2018. Planning of truck platoons: A literature review and directions for future research. {\em Transportation Research Part B} {\bf 107} 212-228.

\item Braiteh, F.-E., F. Bassi, R. Khatoun. 2025. Platooning in commercial vehicles: A review of current solutions, standardization activities, cybersecurity, and research opportunities. {\em IEEE Transactions on Intelligent Vehicles} {\bf 10}(5) 3134-3155.

\item Coppola, A., D.G. Liu, A. Petrillo, S. Satini. 2022. Eco-driving control architecture for platoons of uncertain heterogeneous nonlinear connected autonomous electric vehicles. {\em IEEE Transactions on Intelligent Transportation Systems} {\bf 23}(12) 24220-24234.

\item Driver Knowledge Test Resource Center. 2026. What is a vehicle platoon? Available at: \\ https://www.driverknowledgetests.com/resources/what-is-a-vehicle-platoon/

\item Future Market Insights, Inc. 2026. Automotive Platooning System Market. Available at: \\ 
https://www.futuremarketinsights.com/reports/automotive-platooning-systems-market

\item Garey, M.R., D.S. Johnson. 1979. {\em Computers and Intractability: A Guide to the Theory of NP-Completeness}, W.~H. Freeman and Company, San Francisco.

\item Guo, S., X. Meng. 2025. Optimal resequencing of connected and autonomous electric vehicles in battery SOC-aware platooning. {\em IEEE Transactions on Transportation Electrification} {\bf 11}(4) 9298-9305.

\item Hussein, A.A., H.A. Rakha. 2020.
Vehicle platooning impact on drag coefficients and energy/fuel saving implications. Available at:
https://arxiv.org/abs/2001.00560.

\item Kaluva, S.T., A. Pathak, A. Ongel. 2020. Aerodynamic drag analysis of autonomous electric vehicle platoons. {\em Energies}, MDPI,  {\bf 13}(15), 1-18.

\item Kellerer, H., U. Pferschy, D. Pisinger. 2010. {\em Knapsack Problems}. Springer, Berlin, Germany.

\item Lammert, M., K. Kelly, K. Walkowicz. 2014. Summary of NREL's recent Class 8 tractor trailer platooning testing. Available at:
https://docs.nlr.gov/docs/fy15osti/62644.pdf

\item Larson, J., K.-Y. Liang, K.H. Johansson. 2015. A distributed framework for coordinated heavy-duty vehicle platooning. {\em IEEE Transactions on Intelligent Transportation Systems}
{\bf 16}(1) 419-429.

\item Larsson, E., G. Sennton, J. Larson. 2015. The vehicle platooning problem: Computational complexity and heuristics. {\em Transportation Research, Part~C} {\bf 60} 258-277.

\item Lee, W.J., S.I. Kwag, Y.D. Ko. 2021. The optimal eco-friendly platoon formation strategy for a heterogeneous fleet of vehicles. {\em Transportation Research, Part~D} {\bf 90} 102664.

\item Li, Q., Z. Chen, X. Li. 2022. A review of connected and automated vehicle platoon merging and splitting operations. {\em IEEE Transactions on Intelligent Transportation Systems} {\bf 23}(12) 22790-22806.

\item Lichtl\'{e}, N., K. Jang, A. Shah, E. Vinitsky, J.W. Lee, A.M. Bayen. 2024.  Traffic smoothing controllers for autonomous vehicles using deep reinforcement learning and real-world trajectory data. Available at: https://arxiv.org/html/2401.09666v1

\item Liu, C., Y. Liu, Z. Xu, L. Wang, H. Tang. 2026. En-route charging and resequencing policies for heterogeneous electric truck platoons. {\em Transportation Research, Part~D} {\bf 154} 105264.

\item McAuliffe, B., M. Lammert, X.-Y. Lu,S.  Shladover, et al. 2018. Influences on energy savings of heavy trucks using cooperative 
adaptive cruise control. SAE International Technical Paper 2018-01-1181,  doi:10.4271/2018-01-1181

\item Peng, C., S. Guo, M. Liu, L. Xiao. 2026. 
Congestion-aware platoon re-sequencing optimization for electric vehicles using deep reinforcement learning. {\em Neurocomputing} {\bf 676} 133040.

\item Pinedo, M.L. 2022. {\em Scheduling; Theory, Systems, and Algorithms}, 6/e. Springer, Berlin, Germany.

\item Rebelo, M., S. Rafael, J.M. Bandeira. 2024. Vehicle platooning: A detailed literature review on environmental impacts and future research directions. {\em Future Transportation} {\bf 4} 591-607.

\item Recurrent Auto. 2023. Slowing down has its benefits: How to game efficiency in an electric car. Available at:
https://www.recurrentauto.com/research/game-efficiency-in-an-electric-car

\item Recurrent Auto. 2024. ICE vs. EV deep dive: How gas cars and electric cars stack up.
Available at: \\
https://www.recurrentauto.com/research/ice-vs-ev-hyundai-kona-edition

\item ResearchNester. 2026. Truck platooning systems market outlook. Available at: \\
https://www.researchnester.com/reports/truck-platooning-systems-market/3874

\item Srisomboon, I., S. Lee. 2021. A sequence change algorithm in vehicle platooning for longer driving range. {\em International Conference on Information Networking (ICOON)} 24-27.

\item Sun, X., Y. Yin. 2019. Behaviorally stable vehicle platooning for energy savings. {\em Transportation Research, Part~C} {\bf 99} 37-52.

\item Tian, Y., K. Xu, N. Ansari. 2005. TCP in wireless environments: Problems and solutions. {\em IEEE Radio Communications}, March, 527-532.


\item van de Hoef, S., K.H. Johansson, D.V. Dimarogonas. 2018. Fuel-efficient en route formation of truck platoons. {\em IEEE Transactions on Intelligent Transportation Systems} {\bf 19}(1) 102-112.

\item Yang, Z.-F., S.-H. Li, A.-M. Liu, Z. Yu, H.-J. Zeng, S.-W. Li. 2019. Simulation study on energy saving of passenger car platoons based on DrivAer model. {\em Energy Sources, Part A: Recovery, Utilization, and Environmental Effects} {\bf 41}(24), 3076-3084.

\item Yang, Z., L. Wang, Z. Yu, H. Wang, W. Sun. 2024. Energy efficient strategy for heterogeneous truck platooning based on a non-uniform platooning model. PMC / Nature Scientific Reports. Available at:
doi: 10.1038/s41598-024-80232-5.

\item Zabat, M., N. Stabile, S. Farascaroli, F. Browand. 1995. The aerodynamic performance of platoons: A final report.

\item Zhang, L., F. Chen, X. Ma, X. Pan. 2020. Fuel economy in truck platooning: A literature overview and directions for future research. {\em Journal of Advanced Transportation}, ID 2604012.

\item Zhang, Y., G. Geffen, J. Zhao, M. Shang, S. Wang, Y.-J. Wu. 2026. Safety, mobility, and environmental impacts of driver-assistance-enabled electric vehicles: An empirical study.
Available at: 
https://arxiv.org/abs/2601.17256

\item Zheng, B., S. Guo. 2025. Dynamic platoon re-sequencing for electric vehicles based on bootstrapped DQN. {\em Electronics} {\bf 14} 4417.

\item Zheng, B., S. Guo, M. Liu, L. Xiao. 2025. Dynamic re-sequencing of EV platoons using noisy dueling DQN for energy fairness. {\em International Conference on Systems, Man and Cybernetics (SMC)}, Vienna, Austria, October.
    
\end{hangref}

\end{document}